\RequirePackage{amsthm}

\documentclass[sn-mathphys, Numbered]{sn-jnl}

\usepackage{graphicx}%
\usepackage{multirow}%
\usepackage{amsmath,amssymb,amsfonts}%
\usepackage{amsthm}%
\usepackage{mathrsfs}%
\usepackage[title]{appendix}%
\usepackage{xcolor}%
\usepackage{textcomp}%
\usepackage{manyfoot}%
\usepackage{booktabs}%

\usepackage{mathtools}
\usepackage{nicefrac}
\usepackage{stackrel}
\usepackage{xspace}
\usepackage{hyperref}
\usepackage[noabbrev,capitalise]{cleveref}
\usepackage{enumitem}
\usepackage{csquotes}
\usepackage{bm}
\usepackage{autonum}
\usepackage{wrapfig}

\newcommand{\RR}{\mathbb{R}}

\DeclarePairedDelimiterX{\inp}[2]{\langle}{\rangle}{#1, #2}

\crefformat{equation}{(#2#1#3)}
\crefrangeformat{equation}{(#3#1#4) to~(#5#2#6)}
\crefmultiformat{equation}{(#2#1#3)}{ and~(#2#1#3)}{, (#2#1#3)}{ and~(#2#1#3)}
\crefname{theorem}{Th.}{Theorems}
\crefname{definition}{Def.}{Definitions}
\crefname{corollary}{Cor.}{Corollaries}
\crefname{lemma}{Lem.}{Lemmata}

\theoremstyle{thmstyleone}%
\newtheorem{theorem}{Theorem}[section]
\newtheorem{proposition}[theorem]{Proposition}%
\newtheorem{lemma}[theorem]{Lemma}%
\newtheorem{corollary}[theorem]{Corollary}%

\theoremstyle{thmstyletwo}%
\newtheorem{example}[theorem]{Example}%
\newtheorem{remark}[theorem]{Remark}%

\theoremstyle{thmstylethree}%
\newtheorem{definition}[theorem]{Definition}%

\begin{document}

\title[Convergence proof for first-order position-based dynamics]{Convergence proof for first-order position-based dynamics: An efficient scheme for inequality constrained ODEs}
%


\author*[1]{\fnm{Steffen} \sur{Plunder}}\email{plunder.steffen.2a@kyoto-u.ac.jp}

\author[2]{\fnm{Sara} \sur{Merino Aceituno}}\email{sara.merino@univie.ac.at}

\affil*[1]{\orgdiv{Institute for the Advanced Study of Human Biology (ASHBi)}, \orgname{KUIAS, Kyoto University}, \orgaddress{\street{Faculty of Medicine Bldg. B}, \city{Kyoto}, \postcode{606-8303}, \country{Japan}}}
\affil[2]{\orgdiv{Faculty of Mathematics}, \orgname{University of Vienna}, \orgaddress{\street{Oskar-Morgenstern-Platz 1}, \city{Vienna}, \postcode{1090}, \country{Austria}}}


\abstract{
NVIDIA researchers have pioneered an explicit method, position-based dynamics (PBD), for simulating systems with contact forces, gaining widespread use in computer graphics and animation. While the method yields visually compelling real-time simulations with surprising numerical stability, its scientific validity has been questioned due to a lack of rigorous analysis.

In this paper, we introduce a new mathematical convergence analysis specifically tailored for PBD applied to first-order dynamics. Utilizing newly derived bounds for projections onto uniformly prox-regular sets, our proof extends classical compactness arguments. Our work paves the way for the reliable application of PBD in various scientific and engineering fields, including particle simulations with volume exclusion, agent-based models in mathematical biology or inequality-constrained gradient-flow models.
}




\maketitle


\section{Introduction}

\subsection*{Motivation}

Physical systems often involve objects that interact via contact forces, prohibiting volume overlap. Whether one models idealized hard spheres, rigid bodies, or solids, the treatment of such forces involves mathematical challenges due to the non-smooth nature of contact forces \cite{Wriggers2006,Acary2008}. The theory of ordinary differential equations (ODEs) is not equipped to handle these abrupt changes in forces. However, under specific conditions, well-posed models do exist, supported by a rich mathematical theory of non-smooth dynamics developed by Moreau \cite{Moreau1977}, Filippov \cite{Filippov1988},  and others \cite{Acary2008,Brogliato2020,Braun2021,Kleinert2022}.

In the context of applications in mathematical biology, our interest lies in first-order dynamics, such as overdamped Newton’s laws involving contacts. The behavior of contact forces in first- and second-order dynamics is fundamentally different. Unlike the brief, impact-like collisions typical of second-order dynamics, first-order dynamics often involve sliding or pushing between objects, particularly in the presence of friction and overdamping. An example is shown in \cref{fig:intro:particles}, which models the motion of particles with volume exclusion in the overdamped regime. Such particle simulations are common in agent-based modeling in mathematical biology \cite{Buttenschoen2020}. More generally, first-order non-smooth dynamical systems also arise in electric circuit modeling \cite{Acary2008,Braun2021}, granular models, and crowd motion \cite{Maury2019}.

\begin{wrapfigure}{r}{0.3\textwidth}
    \centering
    \includegraphics[width=0.3\textwidth,trim=2cm 0.5cm 2cm 1cm]{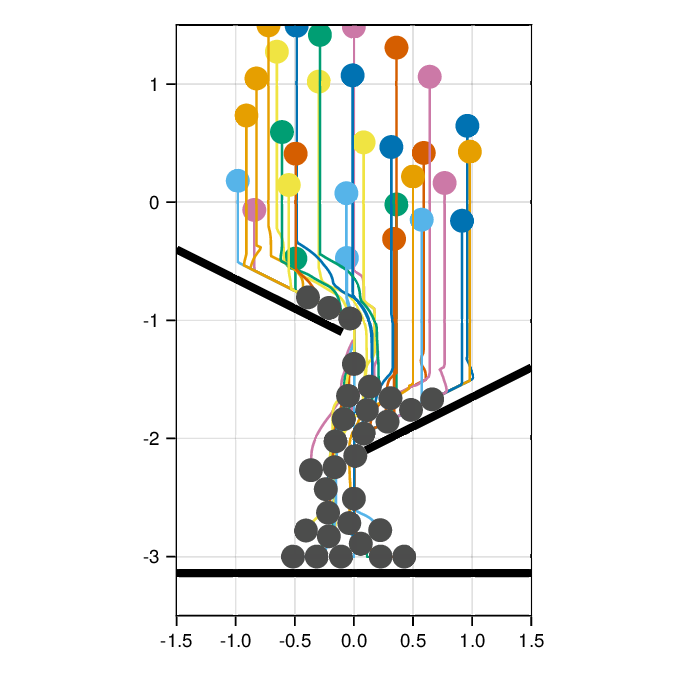}
    \caption{Example for a particle simulation with PBD by imposing the volume exclusion constraint $\Vert X_i - X_j \Vert \geq 2R$ for center positions $X_i \in \mathbb{R}^2$ and radius $R > 0$.}
    \label{fig:intro:particles}
\end{wrapfigure}

Three main strategies exist for the numerical treatment of non-smooth dynamics \cite{Wriggers2006,Acary2008,Dubois2018}: event-tracking, time-stepping, and penalized-constraint schemes \cite{Brogliato2002}. While these methods offer various benefits, they come with significant computational costs. Event-tracking methods and time-stepping methods typically compute the contact forces by solving an implicit equation. However, it is complicated to construct higher-order schemes \cite{Schindler2014}, and the non-smoothness of the dynamics enforces relatively small time steps, creating a computational bottleneck. Penalized-constraint methods smooth out the non-smooth dynamics, making classical ODE methods applicable. Yet, these methods require a delicate balance of numerical stiffness and precision with respect to contact forces. This also results in a computational bottleneck caused by the contact forces.

To address these computational challenges, researchers at NVIDIA have introduced a numerical technique known as position-based dynamics (PBD) \cite{Mueller2006, Bender2017, Macklin2016, Macklin2019}. Unlike conventional time-stepping methods, PBD is explicit rather than implicit, as it terminates an implicit solver for the contact forces right after the first iteration. The reasoning is that in many applications, the best use of the computational budget is to use a computationally cheap approximation of the contact forces but take smaller time steps. This strategy has demonstrated competitive performance in terms of accuracy, stability, speed, and simplicity. PBD's computational efficiency enables real-time simulation of complex mechanical systems, making it widely applicable to fluids, rigid bodies, cloths, and elastic materials \cite{Macklin2013,Macklin2014,Macklin2016,Bender2017,Mueller2020}. PBD serves as a fundamental solver in NVIDIA's multiphysics engine PhysX and the Omniverse simulation platform \cite{Macklin2014,NVIDIA2018,NVIDIA2020}.  

However, despite its growing popularity and broad application scope \cite{Bender2017,Camara2016,Weiss2017,Francu2017,Pecol2011,MonjiAzad2023}, there has been a noticeable gap in providing a rigorous mathematical foundation for PBD. This lack of formal verification has limited its use in scientific fields where the accuracy and reliability of numerical methods are crucial for their adoption.

\subsection*{Mathematical background and projective formulation of PBD}

In this article, we will analyze the convergence of PBD for first-order differential inclusions. Differential inclusions are differential equations with set-valued vector fields, which is an analytically powerful approach to deal with non-smooth terms such as contact forces. We will introduce differential inclusions later in \cref{sec:prelim}. To present the core findings in this introduction, we employ the less abstract but equivalent concept of dynamic complementarity systems \cite{Brogliato2006}.

Dynamic complementarity systems emerge when inequality constraints are incorporated into differential equations using Lagrangian multipliers. To formulate such a system, one requires a vector field $f: \mathbb{R}^d \to \mathbb{R}^d$ and constraint functions $g_1,\dots,g_m : \mathbb{R}^d \to \mathbb{R}$ which introduce the conditions $g_j(x) \geq 0$.  The resulting dynamic complementarity system reads
\begin{align} \label{eq:intro:compl_dyn}
    \dot{x} &= f(x) + \sum_{j=1}^m \lambda_j \nabla g_j(x),   
    \\ 
    \label{eq:intro:compl_cond}
    g_j(x) & \geq 0, 
    \quad \lambda_j \geq 0, 
    \quad g_j(x) \lambda_j = 0
    \quad \text{for all } 1 \leq j \leq m,
\end{align}
where $\dot{x}$ is the time-derivative and $\lambda_1,\dots,\lambda_m$ denote the Lagrangian multipliers for the constraints $g_j(x) \geq 0$. The overdamped motion of particles with volume exclusion serves as an illustrative example, as shown in \cref{fig:intro:particles}.

To solve dynamic complementarity systems numerically, one often couples approximations of \cref{eq:intro:compl_cond} with discretizations of \cref{eq:intro:compl_dyn}. The details of deriving time-stepping methods can be found in \cite{Acary2008}. A useful representation of these methods involves orthogonal projections, simplifying the notation and facilitating analysis. This leads us to define feasible sets $S_j$ and $S$ as follows:
\begin{align}
    S_{j} \coloneqq \{ x \in \mathbb{R}^d \mid g_j(x) \geq 0 \}
    \quad \text{and} \quad 
    S = \bigcap_{j=1}^m S_j.
\end{align}
In this article, the general assumption is that computing $P_{S_j}$ is relatively inexpensive. This tends to hold true especially when the constraint functions $g_j$ resemble distance functions, as is often the case with particles subject to volume exclusion. In situations where the projection is not explicitly available, one could employ a single iteration of a local Gauss-Newton method to approximate it \cite{Deuflhard2015}. However, this approach falls outside the scope of our current discussion for the sake of simplicity.

A fundamental time-stepping method for non-smooth dynamics is the Moreau-Euler method \cite{Moreau1999}
\begin{align}
    x_{k+1} = P_S(x_k + h f(x_k))
    \label{eq:intro:moreau}
\end{align}
where $h > 0$ is the time-step size and $x_k$ is the state at the previous time-step.
Although the method is commonly used to prove the well-posedness of \cref{eq:intro:compl_dyn,eq:intro:compl_cond}, practical computation demands approximations of the projection $P_S$ which is difficult, especially when the sets $S_j$ are non-convex.

An existing approach for approximating projections onto $S = \bigcap_{j=1}^m S_j$ goes back to von Neumann's alternating projection scheme, which is related to the nonlinear projected Gauss-Seidel method (NPGS). Using projections, the resulting time-stepping scheme reads
\begin{align}
    x_{k+1} = \Phi_h^{\mathrm{NPGS}}(x_k) = \left(P^{\mathrm{itr}}_S \right)^{n_k}(x_k + h f(x_k) ) 
    \quad 
    \text{where}
    \quad
    P^{\mathrm{itr}}_S \coloneqq P_{S_m} \circ \dots \circ P_{S_1}.
\end{align}
Here, $n_k$ is decided by a stopping criterion in each iteration. Such a criterion is crucial for obtaining estimates of projection errors, which is a key component for proving convergence \cite{Wu2020}.

\begin{wrapfigure}{r}{0.3\textwidth}
    \centering
    \includegraphics[width=0.28\textwidth]{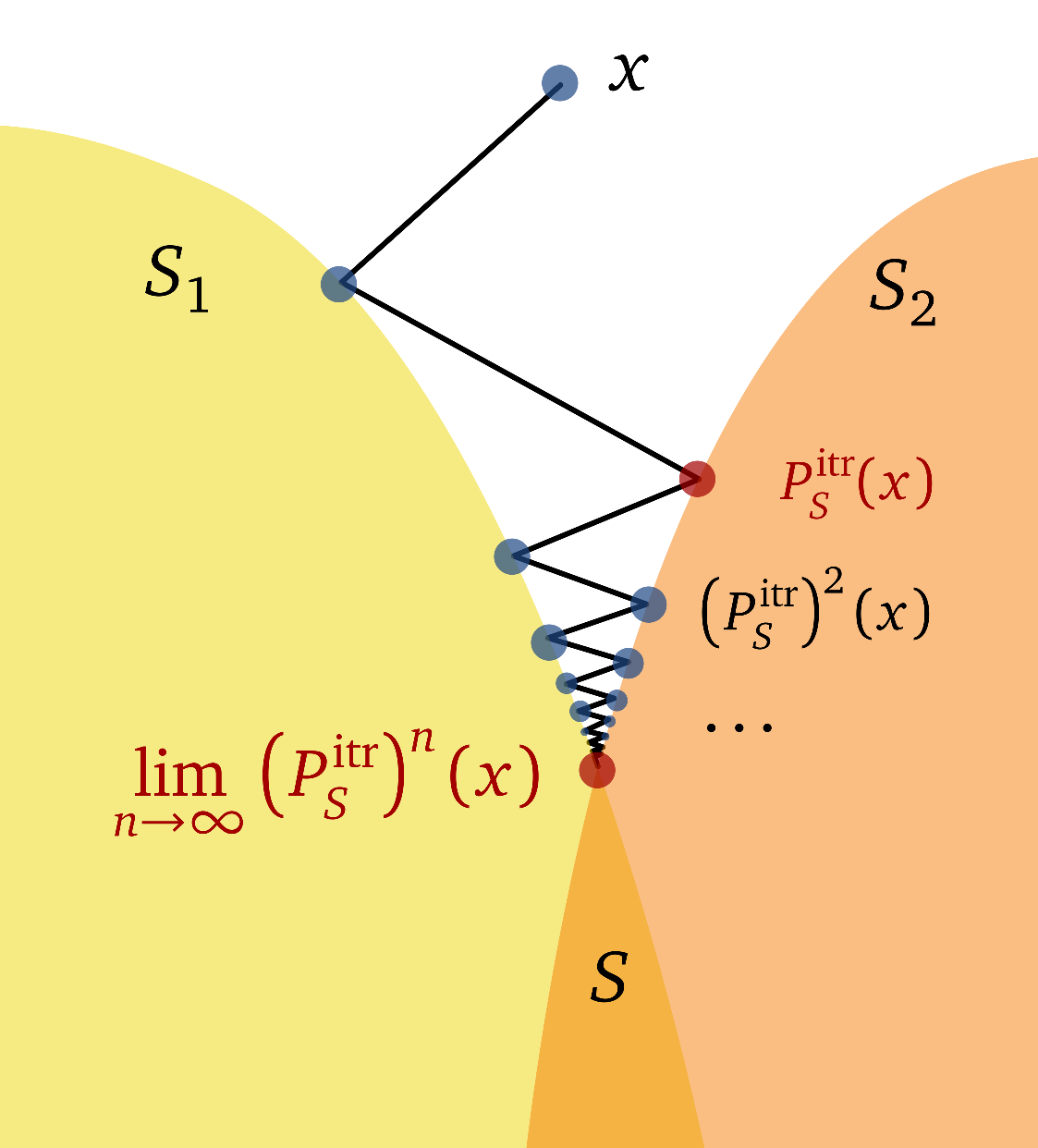}
    \caption{Comparison between one step of the alternating projections ($P_S^{\mathrm{itr}})(x)$ and the limit of repeated application. Unlike the NPGS method, the PBD uses only repetition, which is faster but less precise.}
    \label{fig:intro:alternating}
\end{wrapfigure}
Position-based dynamics (PBD) emerges as a specific form of NPGS where the iteration is truncated after one step ($n_k = 1$). Mathematically, this is described as:
\begin{align}
    x_{k+1} = \Phi_h^{\mathrm{PBD}}(x_k) = P_{S_m} \circ \dots \circ P_{S_1}(x_k + h f(x_k)).
    \label{eq:intro:PBD}
\end{align}

The choice of early termination of NPGS compromises existing convergence proofs since the error estimates afforded by the stopping criterion are no longer applicable. As demonstrated in \cref{fig:intro:alternating}, the error $\Vert P_S^{\mathrm{itr}}(x) - P_S(x) \Vert$ can indeed be significant when using this truncated approach.

\subsection*{Main results}

Our main result, as stated in \cref{thm:pbd}, establishes the uniform convergence of the numerical approximation using position-based dynamics (PBD) as per \cref{eq:intro:PBD}. 
Denoting the exact solutions of  \cref{eq:intro:compl_dyn,eq:intro:compl_cond} as $x: [0,T] \to \mathbb{R}^d$, we prove for $T>0$ the global convergence
\begin{align} 
\sup_{1 \leq k \leq \frac{T}{h}} \Vert x_{k+1} - x(kh) \Vert \to 0
\quad 
\text{as } h \to 0.
\end{align}
The foundational assumptions that make this result possible are discussed in depth in \cref{sec:assumptions}. Notably, these assumptions mirror typical conditions essential for the well-posedness of the complementary system as defined by \cref{eq:intro:compl_dyn,eq:intro:compl_cond}.

In our numerical analysis, we introduce a new notion of numerical consistency specifically tailored for PBD. Our definition just requires convergence in the sense of scalarly upper semicontinuity \cite{Edmond2006,Adly2016}. Leveraging this concept, we establish the numerical consistency of PBD by examining the properties of projection directions and their associated normal cones, as elucidated in \cref{cor:subdiff_sum_susc}.

The cornerstone to proving the stability of PBD  (\cref{lem:iter-proj_error})  is the new estimate 
\begin{align}
    \Vert P_S(x) - P_S^{\mathrm{itr}}(x) \Vert \leq L_P \Vert P_S(x) - x \Vert  
    \quad \text{for all } x \in \mathbb{R}^d \text{ with } d_S(x) < \varepsilon
\end{align}
where $L_P < 1$ is a constant, $d_S(x)$ denotes the distance from $x$ to the set $S$ and $\varepsilon > 0$ is a constant depending exclusively on the sets $S_1,\dots,S_m$. This estimate differs from existing results as in \cite{Luke2008,Lewis2009,Ye2021} by avoiding the typical three-point recursion argument, making it more generally applicable, in particular for proving the stability of PBD.

Furthermore, we adapt the well-posedness proof from \cite{Edmond2006,Bernicot2010} to show that our notion of numerical consistency and stability results in uniform convergence to the exact solution, as stated in \cref{thm:abstract_convergence}. This convergence theorem applies in particular to PBD, but it also extends to methodologies like the extended PBD (XPBD) \cite{Macklin2016} and potentially other related schemes.

In comparing our work with the PBD formulation commonly used in computer graphics \cite{Mueller2006,Macklin2019}, it is crucial to note two important modifications. First, our focus narrows to first-order dynamics, a choice driven by our interest in mathematical biology applications. We explicitly state that our results do not include second-order systems. Second, we operate under the assumption that the projections $P_{S_j}$ are computationally trivial, which simplifies the mathematical framework by eliminating the need to account for additional errors in approximating $P_{S_j}$. Notably, this assumption is particularly relevant when the constraints essentially function as distance functions $d_{S_j}$, which is the case for particle simulations with volume exclusion. In such cases, the projection $P_{S_j}$ coincides with a single iteration of the local Gauss-Newton method as in \cref{eq:intro:PBD}.

\subsection*{Structure of the article}

In \cref{sec:prelim}, we summarize the basic theory of differential inclusions (\cref{def:differential_inclusion}) on prox-regular sets (\cref{def:prox-regularity}), including existence theory (\cref{thm:abstract_existence}). In  \cref{sec:assumptions}, we summarise the assumptions needed for the convergence of PBD. In \cref{sec:numerical_analysis}, we introduce our notion of numerical consistency (\cref{def:consistency}) and stability (\cref{def:stability})  and state an abstract theorem for numerical methods for differential inclusions that guarantees convergence when the method is consistent and stable (\cref{thm:abstract_convergence}). The proof of this theorem is postponed to \cref{app:conv_proof}: the proof follows classical steps taken from \cite{Edmond2006,Bernicot2010}, but is presented for completeness.  \cref{sec:main_proofs} contains the main results, namely, that the PBD method is numerically consistent and stable ( \cref{th:consistency_stability_PBD}), which implies uniform convergence of PBD (\cref{thm:pbd}) by  \cref{thm:abstract_convergence}. 
We showcase the convergence in two numerical experiments in \cref{sec:num_tests}. 
Finally, \cref{app:proof_discs_metric_calm} shows that the assumption of metric calmness is satisfied for the example of particles with volume exclusion.

\section{Preliminaries}
\label{sec:prelim}

In this section, we collect the necessary mathematical framework for the analysis of differential inclusions on prox-regular sets.

We use the following notation: For a set $A \subset \mathbb{R}^d$, we denote the point-set distance to $x\in \mathbb{R}^d$ as $d_A(x) \coloneqq \inf_{a \in A} \Vert x - a \Vert,$ and we denote the projection as $P_A(x)~\coloneqq~\mathop{\mathrm{arg\,min}}_{a \in A} \Vert x - a \Vert$. In general, the projection is a set-valued map $P_A : \mathbb{R}^d \to 2^A$ where $2^A$ denotes the power set.
The ball around zero with radius $r>0$ is $B_r \coloneqq \{ x \in \mathrm{R}^d \mid \Vert x \Vert < r\}$ and we use the Minkowsky sum for the summation of sets, e.g., $A + B \coloneqq \{ a + b \mid a \in A, b \in B\}$.

We consider dynamics with a constraint of the form $x \in S$ where $S = \bigcap_{j=1}^m S_j$ for sets $S_1,\dots,S_m \subset \mathbb{R}^d$. As mentioned in the introduction, in many applications, the feasible set $S$ is not convex. However, for important cases, the sets $S_j$ and $S$ are uniformly prox-regular, which is a generalization of convexity, see the following definition and \cref{fig:prox_regular}.

\begin{figure}[h]
    \centering
    \includegraphics[width=0.5\linewidth]{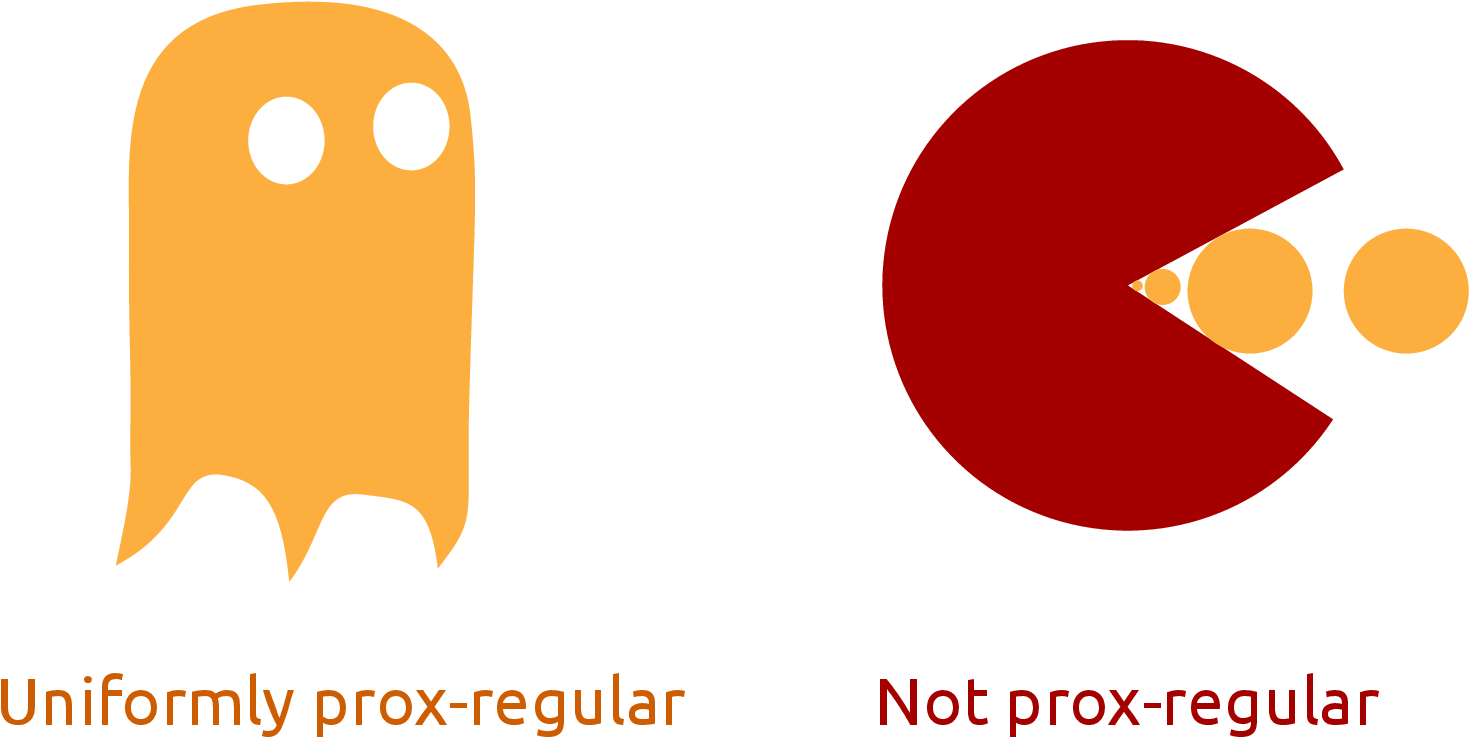}
    \caption{Example for a uniformly prox-regular set (left) and a set which fails to be prox-regular (right). The second set fails to have unique projections for points close to the center. Generally, prox-regular sets cannot have corners with an angle larger than $180^\circ$ (measured from the inside of the domain).}
    \label{fig:prox_regular}
\end{figure}

\begin{definition}[Uniform prox-regular sets and proximal normal cones]
\label{def:prox-regularity}
    Given $\eta > 0$, a set $S \subseteq \mathbb{R}^d$ is called $\eta$-prox-regular, if projections are unique over $S + B_\eta$, i.e., 
\begin{align}
P_S(x) \text{ is single-valued for all } x \in S + B_{\eta}.
\end{align}
    If a set $S$ is $\eta$-prox-regular for some $\eta > 0$, we also call it uniformly prox-regular.
Moreover, we define the proximal normal cone of a uniformly prox-regular set $S$ as
    \begin{align}
    N(S,x) = \{
    v \in \mathbb{R}^d \mid x \in P_S(x + \alpha v ) \text{ for some } \alpha > 0
    \}.
\end{align}
\end{definition}
An intuitive condition for $\eta$-prox-regularity is called the rolling-ball condition, which states that one could roll a ball radius smaller  $\eta$ the entire boundary of the set without it touching two points of the set in any instance \cite{Venel2011}. 

We note that for a non-empty, $\eta$-prox-regular set $S$ and $x \in S + B_{\eta}$ the definition of proximal normal cones implies that
\begin{align}\label{eq:x_Psx_normalcone}
        x - P_S(x) \in N(S, P_S(x))
\end{align}
where the projection is single-valued due to the prox-regularity.

\begin{figure}[ht]
    \centering
    \includegraphics[width=0.6\linewidth]{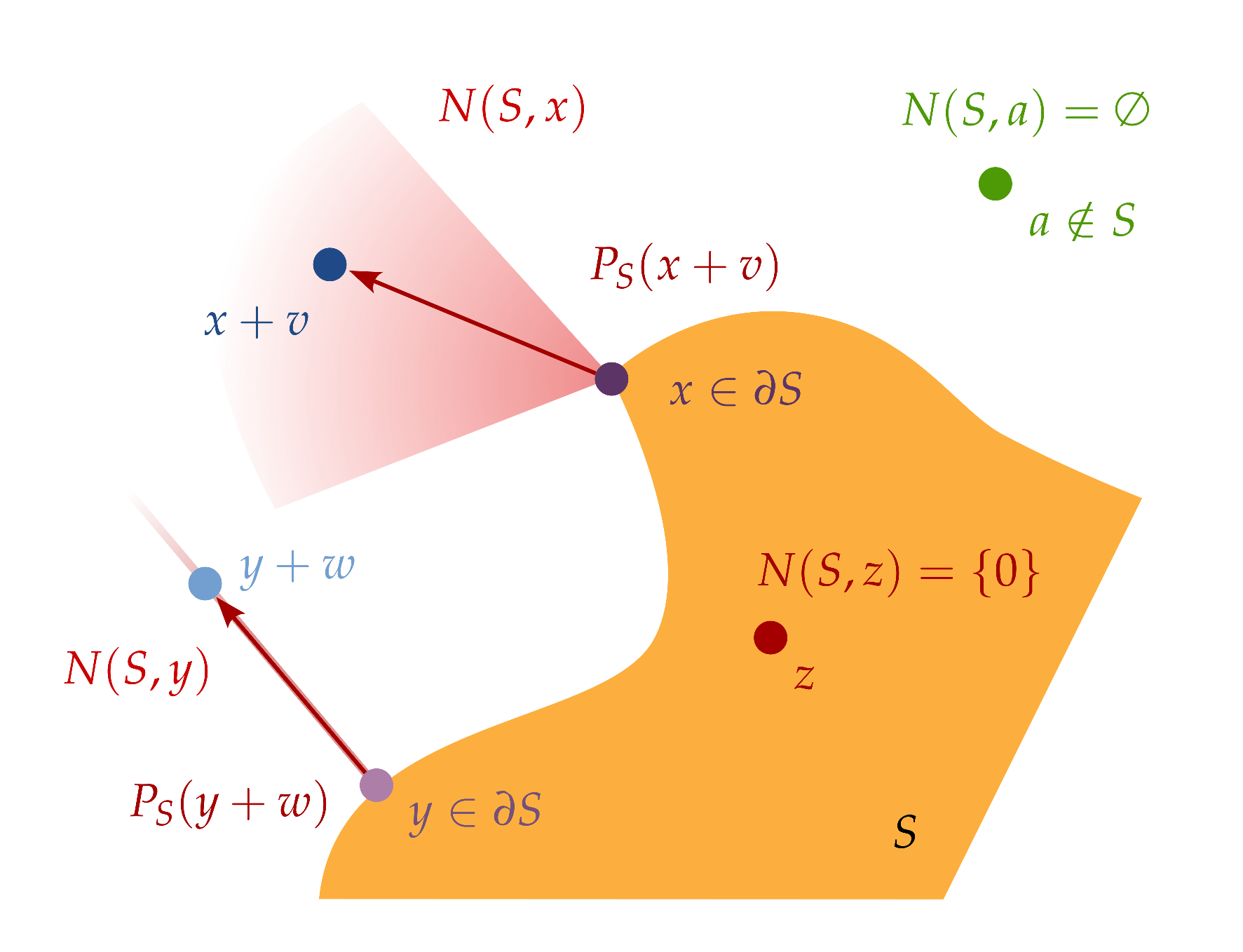}
    \caption{Examples for proximal-normal cones. Notice that the normal cone outside of $S$ is empty, and in the interior of $S$, the normal cone becomes trivial. The non-trivial cases take place only for points at the boundary of the domain $\partial S$. In all cases, the proximal normal cone forms a closed, convex cone. For boundary points where the domain is locally smooth, the proximal normal cone has just one dimension (half-line)---see point $y$ in the figure.}
    \label{fig:nc}
\end{figure}

Conditions for the prox-regularity of sets $S_j$ corresponding to the inverse image of a function (such as $g_j^{-1}([0,\infty))$) and finite intersections of sets ($S = \cap_j S_j$) are studied, for example in \cite[Section 7]{Bernard2011}. One condition that ensures prox-regularity is called metric calmness. For example, the intersection $S = \bigcap^m_j S_j$ is called \emph{metrically calm} at all points of $x \in S$, if  
there exists a constant $\gamma > 0$ (uniform in $x$) and an neighbourhood $U=U(x)$ of $x$ such that
\begin{align}
    d_S(y) \leq \gamma \sum_{j=1}^m d_{S_j}(y) \quad \text{for all } y \in U(x).
    \label{eq:def:intersection_metric_calmness}
\end{align}
If the sets $S_j$  are all $\eta$-prox-regular and the intersection is metrically calm with a constant $\gamma$ uniform in $x$, then $S$ is $\eta'$-prox-regular for $\eta' = \frac{\eta}{m \gamma}$ \cite[Theorem 7.4]{Bernard2011}. A similar statement holds for inverse images. For  details, we refer to \cite{Bernard2011} and \cite{Adly2016}. 

We note that normal cones are not preserved with respect to intersections. In general, one only has the following relation.
\begin{theorem}[{\cite[Theorem 6.42]{Rockafellar1998}}]
\label{thm:nc_intersection}
Let $A = \bigcap_{j=1}^m A_j$ for closed sets $A_j \subset \mathbb{R}^{d}$ , let $x \in \mathbb{R}^d$, then it holds that 
\begin{align}
    \sum_{j=1}^m N(A_j, x) \subseteq N(A, x).
\end{align}
\end{theorem}

A concrete example of prox-regular sets which is used in various applications is that of non-overlapping spheres. 

 \begin{example}[Uniform prox-regularity of non-overlapping spheres]
    Consider $N>0$ spheres with radius $R>0$ and center $X_i \in \RR^d$, $i=1,\dots, N$. We define the feasible set $S$ of non-overlapping spheres as the intersection of the sets
    \begin{align}
    S_{ij} \coloneqq \{ (X_1,\dots,X_N) \in \mathbb{R}^{dN} \mid \Vert X_i - X_j \Vert \geq 2R \},
  \end{align}
that is
  \begin{align}
  S \coloneqq \bigcap_{i < j} S_{ij}.
  \end{align}
This intersection is metrically calm, as we prove for completeness in \cref{lem:spheres_metric_calm} and the sets $S_{ij}$ are $R$-prox-regular. Hence, $S$ is also uniformly prox-regular. Another proof for the uniform prox-regularity of this set is given in \cite[Proposition 4.5]{Venel2011} for the case $d=2$ (but without using metric calmness but an explicit formula for the prox-regularity coefficient).
\end{example}

The notion of prox-regularity allows us to reformulate dynamic complementarity problems as differential inclusions. The following definition introduces the corresponding notion of solutions.

\begin{definition}
\label{def:differential_inclusion}
Let $S$ be a subset of $\RR^d$.
For $T>0, x_0 \in S$ and $f: \mathbb{R}^d \to \mathbb{R}^d$,  we call $x: [0,T] \to \mathbb{R}^d$ a solution of the differential inclusion
\begin{align}
\dot x \in f(x) - N(S,x), \quad x(0) = x_0,  \label{eq:differential_inclusion}
\end{align}
if $x$ is absolutely continuous and satisfies
\begin{align}
    \dot x(t) &\in f(x(t)) - N(S,x(t)) \quad \text{for a.e. } t \in [0,T], \label{eq:def:main_di} \\ 
    x(0) &= x_0. \label{eq:def:main_di_init} 
\end{align}
\end{definition}

\subsection{Well-posedness of differential inclusions on prox-regular sets}

Under the right assumptions, one can show that differential inclusions \cref{eq:differential_inclusion} and dynamic complementarity problems \cref{eq:intro:compl_dyn,eq:intro:compl_cond} have exactly the same solutions \cite[Theorem 9.3.]{Adly2016}.
The advantage of differential inclusions is the rich theoretical framework from variational analysis, which allows to show well-posedness.
\begin{theorem}[Well-posedness of differential inclusions on uniformly prox-regular sets \cite{Edmond2006}]
\label{thm:abstract_existence}
Let $f$ be Lipschitz continuous, $S \subseteq \mathbb{R}^d$ be closed, non-empty and uniformly prox-regular, $T > 0$ and $x_0 \in S$,  then the differential inclusion \cref{eq:def:main_di,eq:def:main_di_init} has a unique absolutely continuous solution.
\end{theorem}

The mathematical framework of this well-posedness result is the foundation for the numerical analysis in this article.  We will recall the analytical properties of the set-valued mapping
\begin{align}
x \mapsto -N(S, x),
\end{align}
which are essential for the theory. A characterizing property of $\eta$-prox-regular sets is that their associated normal cones are hypomonotone, that is, for all $x, y \in S$ we have
\begin{align}
    -\langle{v - w},{x - y}\rangle \leq \frac{\Vert v \Vert + \Vert w \Vert}{2 \eta} \Vert x - y \Vert^2
    \quad \text{for all $v \in N(S, x)$, $w \in N(S, y)$}.
    \label{eq:prelim:nc:hypermono}
\end{align}
For the special case $\eta = \infty$, we obtain $\langle{v - w},{x - y}\rangle \geq 0$, which is a characterization of convexity. In this sense, $\eta$-prox-regular sets generalise convexity.

For analysis involving normal cones, it is often sufficient to consider only vectors with bounded length, e.g., $\Vert v \Vert \leq 1$. For a uniformly prox-regular set $S$ and $x \in \mathbb{R}^d$, the intersection $N(S,x) \cap B_1$ is exactly the proximal subdifferential of the distance function, i.e.,
\begin{align}
    \partial^P d_S(x) = N(S,x) \cap B_1.
\end{align}
In this article, we will use this relation simply as a convenient shorthand notation. For the definition and analysis of proximal subdifferentials, we refer to \cite{Rockafellar1998} or \cite{Thibault2022}.

By \cref{eq:prelim:nc:hypermono}, the map $x \mapsto -\partial^P d_S(x)$ is one-sided Lipschitz continuous with constant $L = \frac{1}{\eta}$, i.e. for all $x, y \in \mathbb{R}^d$ we have
\begin{align}
    \langle v - w, x - y \rangle \leq L \Vert x - y \Vert^2
    \quad 
    \text{for all } v \in -\partial^P d_S(x) \text{ and } w \in -\partial^P d_S(y).
\end{align} 
This property leads to the uniqueness of solutions for \cref{eq:def:main_di}. 

For the existence of solutions, it is sufficient to require upper semicontinuity of $x \mapsto -\partial^P d_S(x)$. However, in this article, we will use the slightly weaker notion of scalarly upper semicontinuity, which turns out to be useful for numerical analysis.

Let us recall that a single-valued function $f: \mathbb{R} \to \mathbb{R}$ is upper semicontinuous at $x \in \mathbb{R}$ if 
\begin{align}
    \limsup_{y \to x} f(y) \leq f(x).
\end{align}
There is also a notion of upper semicontinuity for set-valued maps. However, we will use the following definition instead.
\begin{definition}[Scalarly upper semicontinuity {\cite{Edmond2006}}]
A set-valued function $\mathcal{F}: \mathbb{R}^d \to 2^{\mathbb{R}^d}$ is scalarly upper semicontinuous at $x \in \mathbb{R}^d$ if for all $v \in \mathbb{R}^d$ it holds
    \begin{align}
        \limsup_{y \to x} \sigma( \mathcal{F}(y), v) \leq \sigma( \mathcal{F}(x), v),
    \end{align}
where $\sigma$ denotes the support function which is defined as 
    \begin{align}
    \sigma(A, v) \coloneqq \sup_{w \in A} \langle w, v \rangle
    \end{align}
    for $A \subset \mathbb{R}^d$ and $v \in \mathbb{R}^d$.
\end{definition}
The support function is a common tool from convex analysis \cite{Rockafellar1998}, one of its properties is that it captures subset relations
\begin{align} \label{eq:monotonicity_support_function}
    A \subseteq B \quad \Leftrightarrow 
    \quad 
    \sigma(A, v) \leq \sigma(B, v) \quad \text{for all } v \in \mathbb{R}^d,
\end{align}
where $A, B \subseteq \mathbb{R}^d$ are convex sets. 

The next proposition shows that the notion of scalarly upper semicontinuity applies in particular to the proximal subdifferentials, see also \cref{fig:susc} for a sketch of the geometric setting.
\begin{proposition}[{\cite[Prop. 3.4 (ii)]{Adly2017}}]
\label{prop:subdiff_susc}
Let $S$ be non-empty, closed and $\eta$-prox-regular and $x \in S$, then 
\begin{align}
    \limsup_{y \to x} \sigma( -\partial^P d_S(y), w ) \leq \sigma( -\partial^P d_S(x), w )
\end{align}
for all $w \in \mathbb{R}^d$. In particular, $x \mapsto -\partial^P d_S(x)$ is scalarly upper semicontinuous.
\end{proposition}

\begin{figure}[h]
    \centering
    \includegraphics[width=0.5\linewidth]{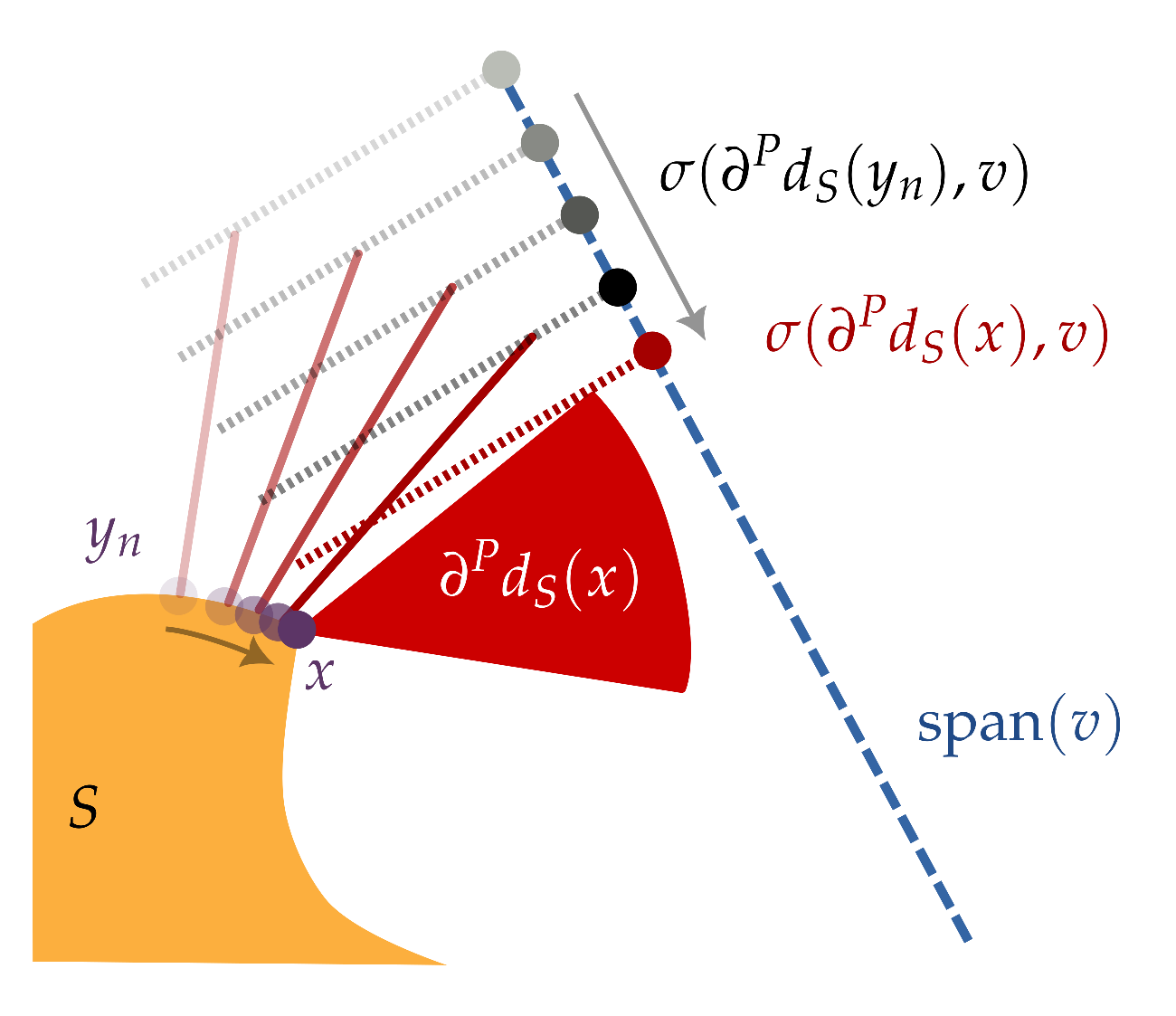}
    \caption{Example of scalarly upper semicontinuity of the proximal subdifferentials of the distance function. The convergence holds if, for all test directions $v$ (blue), the corresponding projections onto $\mathrm{span}(v)$ converge upper semicontinuously.}
    \label{fig:susc}
\end{figure}

\subsection{Assumptions}
\label{sec:assumptions}

In this article, we will take the following assumptions.

\begin{itemize}
    \item We assume that the right-hand side $f : \mathbb{R}^d \to \mathbb{R}^d$ is Lipschitz continuous and bounded, i.e. there exist constants $L_f > 0$ and $M_f > 0$ such that
\begin{align}
    \Vert f(x) - f(y) \Vert \leq L_f \Vert x - y \Vert 
    \tag{H1} \label{ass:f:lipschitz}
\end{align}
and 
\begin{align}
\Vert f(x) \Vert \leq M_f    
    \tag{H2} \label{ass:f:bounded}
\end{align}
for all $x, y \in \mathbb{R}^d$.
\item We assume for the feasible set $S \subset \mathbb{R}^d$  that
\begin{align}
    S \quad \text{is $\eta$-prox-regular, closed, and non-empty.}
    \tag{H3} \label{ass:S_prox-reg}
\end{align}
\item Moreover, there exists sets $S_1,\dots,S_m \subset \mathbb{R}^d$ such that 
\begin{align}
    S &= \bigcap_{j=1}^M S_j \quad \text{and} \tag{H4} \label{ass:S_Sj}\\ 
    S_1, \dots, S_m &\quad \text{are all $\eta$-prox-regular, closed and non-empty.} \tag{H5} \label{ass:Sj_prox-reg}
\end{align}
\item Metric calmness of the intersection: We require that there exists a constant $\alpha > 0$, such that 
\begin{align}
\inf_{x \in \mathbb{R}^d}
\max_{j = 1,\dots,m} \frac{d_{S_j}(x)}{d_S(x)} \geq \alpha.
\tag{H6}
\label{ass:metric_calm}
\end{align}
\end{itemize}

Our definition of metric calmness of the intersection is adapted from the convergence proof, but the condition is equivalent to uniform metric calmness as in \cref{eq:def:intersection_metric_calmness}.

We note that most of these assumptions coincide with the assumptions needed for well-posedness as in \cref{thm:abstract_existence}. The additional requirement of metric calmness \cref{ass:metric_calm} is not directly required for well-posedness since uniform prox-regularity of $S$ would suffice. However, metric calmness is a so-called constraint quantification condition, which also ensures that constraint systems as in \cref{eq:intro:compl_dyn,eq:intro:compl_cond} have sufficiently well-behaved normal cones. In terms of constraint functions, metric calmness is called subregularity, and it is a common constraint quantification condition \cite{Dontchev2014,Rockafellar1998,Hesse2013}.

\subsection{Numerical analysis for differential inclusions}
\label{sec:numerical_analysis}

To describe generic numerical one-step methods for the differential inclusion \cref{eq:def:main_di}, we consider their numerical flow map $\Phi_h : \mathbb{R}^d \to \mathbb{R}^d$ where $h$ denotes the time-step size. For given
initial condition $x_0 \in S$, terminal time $T > 0$ and $n \in \mathbb{N}$, we set the time-step size as $h = \frac{T}{n}$ and define the time stepping scheme as $x_{k}^{(n)} \coloneqq \Phi_h(x_{k-1}^{(n)})$ with time-steps $t_k \coloneqq hk$. The piecewise linear interpolation through the discrete time-steps yields a continuous function $x^{(n)} \in C([0,T],\mathbb{R}^d)$ which is defined by 
\begin{align}
x^{(n)}(t) \coloneqq
\frac{t - t_k}{h} x^{(n)}_{k+1}  + \frac{t_{k+1} - t}{h} x^{(n)}_k
\quad 
\text{for } t \in [t_k, t_{k=1}) \text{ and } 0 \leq k < n.
\label{eq:num_traj}
\end{align}

The following definitions provide the basis for the convergence proof of the PBD method. However, the convergence theorem we use is not restricted to the PBD method and is possibly applicable to other schemes as well.

\begin{definition}[Stability]
\label{def:stability}
We say that a numerical method has \emph{bounded constraint violation}, if there exists some $K > 0$, such that
    \begin{align}
    d_S( \Phi_h(x) ) < K h \quad \text{holds for all } x \in  S + B_{K h},\label{eq:def:feasible}
    \end{align}
provided $h > 0$ is sufficiently small.

A numerical method is \emph{stable} if it has bounded constraint violation and if for all sufficiently small $h > 0$ and any $1 \leq k \leq \frac{T}{h}$, it holds
    \begin{align}
        \Vert \Phi_h(x_k) - x_k \Vert \leq h ( A + B \Vert x_k \Vert ) \quad \mbox{for some } A, B>0,
        \label{eq:def:stability}
    \end{align}
where $x_k \coloneqq \left( \Phi_h \right)^k( x_0 )$ and $x_0 \in S$.
\end{definition}

\begin{definition}[Scalarly upper semicontinuous consistency]
\label{def:consistency}
Let us denote $h^{(n)} = \frac{T}{n}$. A numerical flow is \emph{scalarly upper semicontinuously consistent} if there exists a constant $C, K > 0$ such that for any $x^* \in S$ and all sequences $y^{(n)} \to x^* \in S$ with $d_S(y^{(n)}) < K h^{(n)}$ it holds that for all $v \in \mathbb{R}^d$
\begin{align}
        \limsup_{n \to \infty} \sigma\left( \left\{ \frac{1}{h^{(n)}} \left( \Phi_{h^{(n)}}(y^{(n)}) - y^{(n)} \right) - f(y^{(n)})\right\}, \, v\right) \leq \sigma(-C \, \partial^P d_S(x^*), \, v ).\label{eq:def:consistency}
    \end{align}
\end{definition}

\begin{remark}
The intuition for the definition \eqref{eq:def:consistency} is explained next. Our focus is to solve numerically the differential inclusion $\dot{x}  \in f(x) - N(S,x)$, which is equivalent to $\dot{x} - f(x) \in N(S,x)$. In PBD, $\dot{x} - f(x)$ is approximated by
\begin{align} \label{eq:approximation_PBD}
\dot{x} - f(x) \approx \frac{1}{h^{(n)}} \left( \Phi_{h^{(n)}}(y^{(n)}) - y^{(n)} \right) - f(y^{(n)}),
\end{align}
and the goal of the consistency relation would be to show that as $n \to \infty$ it holds that this term belongs to $ N(S,x^*)$ with $x^*$ begin the limit of the sequence $y^{(n)}$. This is shown in a weak sense using the support function $\sigma$ through the estimate \eqref{eq:def:consistency}. Indeed, by \eqref{eq:monotonicity_support_function}, we know that $A\subset B$ if and only if $\sigma(A,v) \leq \sigma(B,v)$ for all $v\in \mathbb{R}$, in particular, this means that in the limit we expect our approximation to belong to the set $-C\partial^P d_S(x^*)$, but on the other hand, we also have that
$$-C\partial^Pd_S(x^*)= -(N(S,x)\cap B_C) \subset -N(S,x).$$
\end{remark}

A theoretical example that satisfies the consistency and stability conditions is the Moreau-Euler method, which has the numerical flow 
\begin{align}
    \Phi_h^{\text{Moreau}}(x) \coloneqq P_S(x + h f(x)).
\end{align}
This method is also used in various existence proofs for solutions of differential inclusions \cite{Adly2017,Bernicot2010}. However, these proofs generalize for general numerical methods, provided that they satisfy the above consistency and stability definitions. 

\begin{theorem}
\label{thm:abstract_convergence}
We assume that \cref{ass:f:lipschitz,ass:S_prox-reg} hold and let $x : [0,T] \to \mathbb{R}^d$ be the unique solution of \cref{eq:def:main_di,eq:def:main_di_init} (given by Theorem \ref{thm:abstract_existence}).

If $\Phi_h$ a numerical flow which satisfies \cref{eq:def:feasible,eq:def:consistency,eq:def:stability},
then, the numerical approximations $x^{(n)} ~:~ [0, T] ~\to ~\mathbb{R}^d$, as defined in \cref{eq:num_traj}, satisfy
\begin{align}
   \sup_{t\in[0,T]} \Vert x^{(n)}(t) - x(t) \Vert \to 0
\end{align}
as $n \to \infty$.
\end{theorem}

The proof is given in \cref{app:conv_proof} in the appendix. The proof is an abstraction of the convergence proof of the Moreau-Euler method as in \cite{Edmond2006}.

\section{Convergence of first-order position-based dynamics}
\label{sec:num_ana}

This section is devoted to the convergence proof of position-based dynamics stated in Th. \ref{th:consistency_stability_PBD}. We first define the numerical flow of first-order PBD and then derive error estimates for the iterative projections, which are the core of the numerical method. Using the error estimates, we can show consistency in a scalarly upper semicontinuous sense and numerical stability. These properties imply convergence by \cref{thm:abstract_convergence}.

\label{sec:main_proofs}

We recall the Moreau-Euler scheme \eqref{eq:intro:moreau}, which is a convergent method by \cref{thm:abstract_convergence}.
The core of the position-based dynamics is to replace the projection $P_S$ in the Moreau-Euler method with the approximation 
\begin{align}
    P^{\mathrm{itr}}_S(q) \coloneqq P_{S_m} \circ \cdots \circ P_{S_1} (q).
\end{align}

\subsection{Estimates for projection errors}

In the following, we consider a starting point $q_0 \in \mathbb{R}^d$ (typically not too far away from $S$), and we define the iterative projection of a point $q_0$ onto the intersection $S = \cap_{j} S_j$ as 
\begin{align}
    q_j &\coloneqq P_{S_j}(q_{j-1})
    \quad 
    \text{for } 1 \leq j \leq m, \label{eq:def:qj}
    \\ 
    P^{\mathrm{itr}}_S(q_0) &\coloneqq q_m. \notag 
\end{align}
Our aim is to provide estimates for $\Vert q_0 - q_m \Vert = \Vert q_0 - P^{\mathrm{itr}}_S(q_0) \Vert$ (stability) and for $\Vert P_S(q_0) - P^{\mathrm{itr}}_S(q_0) \Vert$ (projection error). 

In the context of optimisation and convex analysis, the iterated projections are also called alternating projections, as most results are formulated for projection onto the intersection of two sets. A classical argument for the convergence of alternating projections relies on a three-point estimate \cite[Lem. 1]{Hesse2013}, which provides local convergence rates \cite[Th. 5.2]{Lewis2009}. However, we are not interested in the limit ${\lim_{n\to\infty} \Vert P_S(q_0) - (P_S^{\mathrm{itr}})^n(q_0) \Vert}$ but instead in the error after the first iteration ${\Vert P_S(q_0) - P^{\mathrm{itr}}_S(q_0) \Vert}$. Therefore, three-point estimates or other asymptotic arguments are not applicable. Instead, we use a new technique to obtain the required bounds.

We also denote the increments and projection errors as
\begin{align}
    v_j &\coloneqq q_j - q_{j-1} = P_{S_{j}}(q_{j-1}) - q_{j-1}, \label{eq:def:vj}\\
    e_j &\coloneqq q_j - P_S(q_0), \label{eq:def:ej}
\end{align}
for $j=1,\hdots,m$ (notice that we are assuming here the projections to be unique as justified by the following lemma).
\begin{remark} \label{rem:errors_and_v}
This leads, in particular, to the following relations, which will be continuously used in the sequel:
\begin{align} \label{eq:errors_and_v}
   e_j = e_{j-1} + v_j \quad \mbox{ and }\quad  -v_{j} \in N(S_j, q_j),
\end{align}
 as visualized in \cref{fig:projection_error} and explained in \eqref{eq:x_Psx_normalcone}.
\end{remark}
\begin{figure}[h]
    \centering
    \includegraphics[width=0.5\linewidth]{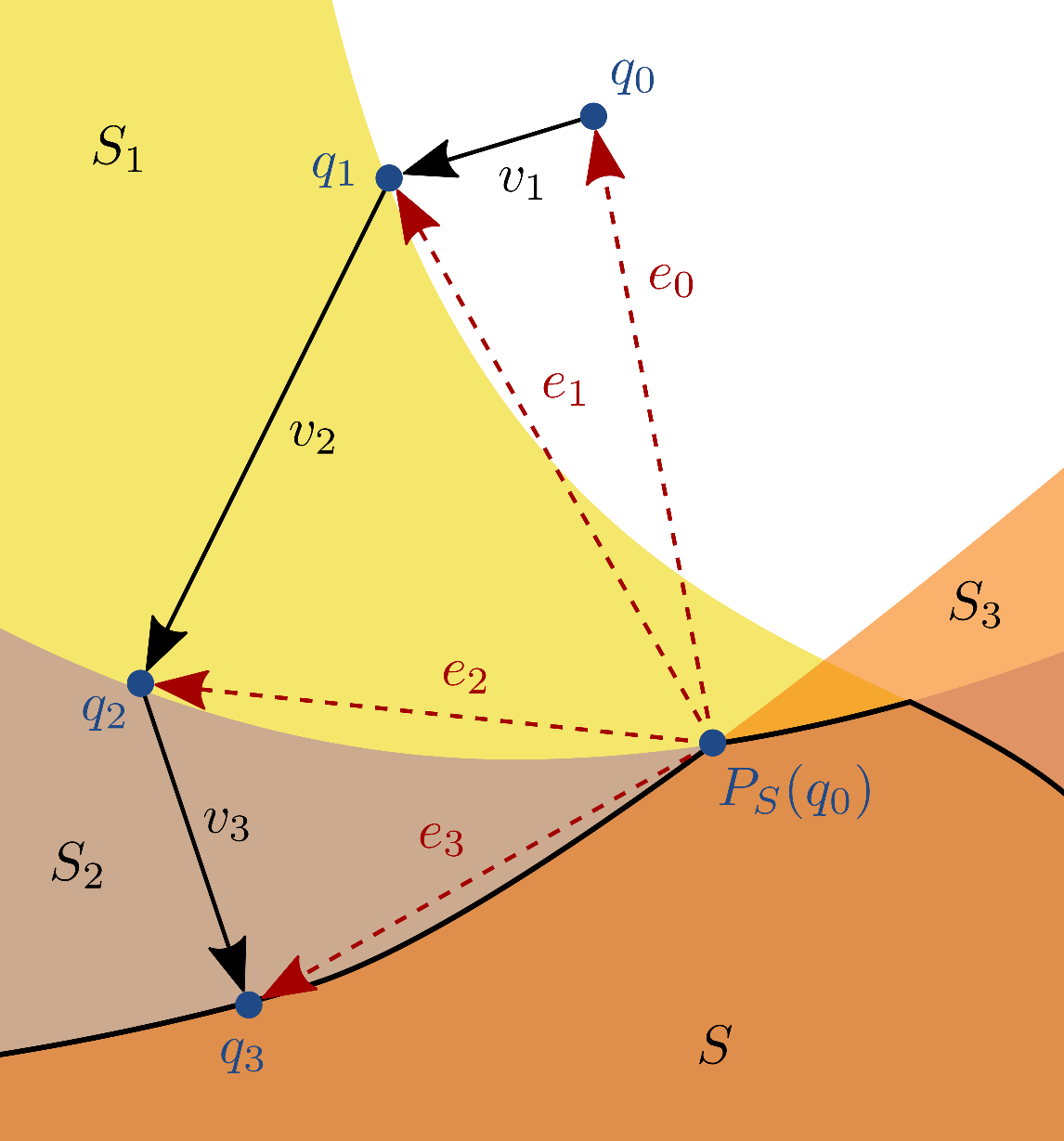}
    \caption{Sketch of iterated projections $q_0 \mapsto q_1 \dots \mapsto q_N = P^{\mathrm{itr}}(q_0)$ and the corresponding projections errors $e_j = q_j - P_S(q_0)$.
    Note that $-v_j \in N(S_j, q_j)$ which is the reason why the PBD method is consistent. However, this property alone is not sufficient to ensure the stability of PBD.}
    \label{fig:projection_error}
\end{figure}

Before we can prove the main error estimate for projections, we need a rough bound to ensure that all projections are single-valued.

\begin{lemma}[Basic error bounds.]
\label{lem:trivial-proj-bound}
Assume that \cref{ass:S_prox-reg,ass:S_Sj} hold for $S = \bigcap_{j=1}^m S_j$ and
let $\varepsilon_{\eta} ~\coloneqq~ 2^{1-m} \eta$. Then, for every $q_0 \in S + B_{\varepsilon}$ with $\varepsilon\leq \varepsilon_\eta$ we have that
\begin{align}
    q_j &\coloneqq P_{S_j}(q_{j-1}) \text{ is unique,} \\ 
    \Vert q_j - q_{j-1} \Vert &\leq 2^{j-1} d_S(q_0), \label{eq:bound_v}\\
    d_S(q_j) &\leq 2^{j} d_S(q_0), \label{eq:bound_d_S}
\end{align}
for all $1 \leq j \leq m$.
\end{lemma}

\begin{proof}
We prove by induction over $j$. For the induction start, we might define $S_0 \coloneqq \mathbb{R}^d$, then, the claim holds trivially for $j = 0$ with $q_{-1} \coloneqq q_0$ since $q_0 = P_{S_0}(q_{-1})$, $\Vert q_{-1} - q_0 \Vert = 0$ and $d_{S}(q_0) = 2^0 d_S(q_{0})$.

    For the induction step, we consider $q_j \in P_{S_j}(q_{j-1})$ (which may not be unique). By the induction assumption, it holds that
    \begin{align}
        \Vert q_j - q_{j-1} \Vert = d_{S_j}( q_{j-1} ) \leq d_S( q_{j-1} ) \leq 2^{j-1} d_S(q_0).
    \end{align}
In particular, this shows that the projection point $q_j$ is indeed unique since $S_j$ is $\eta$-prox-regular and
\begin{align}
    d_{S_j}(q_{j-1}) < 2^{j-1+1-m} \eta \leq\eta.
\end{align}
Moreover,  we have
\begin{align}
    d_{S}(q_j) \leq d_S(q_{j-1}) + \Vert q_{j} - q_{j-1} \Vert \leq 2^{j} d_S(q_0),
\end{align}    
which concludes the claim.
\end{proof}

\begin{lemma}[Single projection error.]
    \label{lem:single_proj}
   We assume that \cref{ass:S_prox-reg,ass:S_Sj,ass:Sj_prox-reg} hold for $S = \bigcap_{j=1}^m S_j$.  
    Let $q_0 \in \mathbb{R}^d$ and let  $e_j$, and $v_j$ be defined as in \cref{eq:def:vj,eq:def:ej} for $j=1,\hdots,m$. 
    Then, if $\Vert v_j \Vert < \eta$ and $\Vert e_{j-1} \Vert \neq 0$, it holds that
    \begin{align} \label{eq:bound_error_single_projection}
        \Vert e_j \Vert^2 \leq \left( 1 - \frac{\Vert v_j \Vert}{\eta} \right)^{-1} \left( 1 - \frac{\Vert v_j \Vert^2}{\Vert e_{j-1} \Vert^2} \right) \Vert e_{j-1} \Vert^2.
    \end{align}
The inequality remains valid if the projections $q_j$ are not unique, e.g., if one replaces \cref{eq:def:qj} with the condition $q_j \in P_{S_j}(q_{j-1})$.
\end{lemma}

\begin{proof}
For the proof, we recall remark \ref{rem:errors_and_v}.
First, applying the Pythagorean identity on $e_{j-1} = e_{j} - v_j$ yields
    \begin{align}
    \Vert e_{j-1} \Vert^2 &= \Vert e_{j} \Vert^2 - 2 \langle{ v_{j} },{ e_{j} }\rangle + \Vert v_j \Vert^2, 
    \end{align}
which is equivalent to
    \begin{align} \label{eq:aux_pyt}
         \quad \Vert e_{j} \Vert^2 &= \Vert e_{j-1} \Vert^2 + 2 \langle{ v_{j} },{ e_{j} }\rangle - \Vert v_j \Vert^2.
    \end{align}

Using the hypomonotonicity of  normal cones \cref{eq:prelim:nc:hypermono}, applied to $-v_j \in N(S_j, q_j)$ and $0 ~\in~ N(S_j, P_S(q_0) )$, we get 
    \begin{align}
        - \langle -v_j - 0 , q_j - P_S( q_0 ) \rangle = \langle v_j, e_j \rangle \leq \frac{ \Vert v_j \Vert }{2 \eta} \Vert q_j - P_S( q_0 ) \Vert^2 = \frac{\Vert v_j \Vert}{2\eta} \Vert e_j \Vert^2.
    \end{align}
Using this last estimate on \eqref{eq:aux_pyt}, we have that
\begin{align}
        \Vert e_j \Vert^2 &\leq \Vert e_{j-1} \Vert^2 + \frac{ \Vert v_j \Vert }{\eta} \Vert e_j \Vert^2 - \Vert v_j \Vert^2,
\end{align}
which implies that (since, by assumption $\Vert v_j \Vert\leq \eta$)
\begin{align}
        \Vert e_j \Vert^2 &\leq \left( 1 - \frac{\Vert v_j \Vert}{\eta} \right)^{-1} ( \Vert e_{j-1} \Vert^2 - \Vert v_j \Vert^2 ),
\end{align}
and we finally conclude the result by factoring out $\Vert e_j \Vert^2$.
We note that the proof does not require the uniqueness of the projections $P_{S_j}(q_{j-1})$.
\end{proof}

\begin{remark} \label{rem:bounds}
    Notice that the previous lemma implies, in particular, that
    $$\Vert v_j\Vert \leq \Vert e_{j-1}\Vert, \quad j=1,\hdots, m,$$
    and if $\Vert e_j\Vert\neq 0$, this inequality is strict.\\
    The key to proving that the 
    projection error $e_m$ decreases compared to $e_0$ is to show that the product of the coefficients
    $$f_j^2:=\left( 1 - \frac{\Vert v_j \Vert}{\eta} \right)^{-1} \left( 1 - \frac{\Vert v_j \Vert^2}{\Vert e_{j-1} \Vert^2} \right)$$
    is strictly smaller than one for initial data $q_0$ close enough to the admissible set $S$, i.e., for $d_S(q_0)$ small. However, notice that the  that the bound \eqref{eq:bound_v} implies that $\Vert v_j \Vert$ tends to zero if $d_S(q_0)\to 0$. And if $\Vert v_j \Vert\to 0$, then $f_j^2\to 1$, which is not a good enough bound to ensure the error reduction. To overcome this, the assumption of metric calmness for the intersection \cref{ass:metric_calm} will be crucial.
\end{remark}


\begin{figure}
    \centering
    \includegraphics[width=0.4\textwidth]{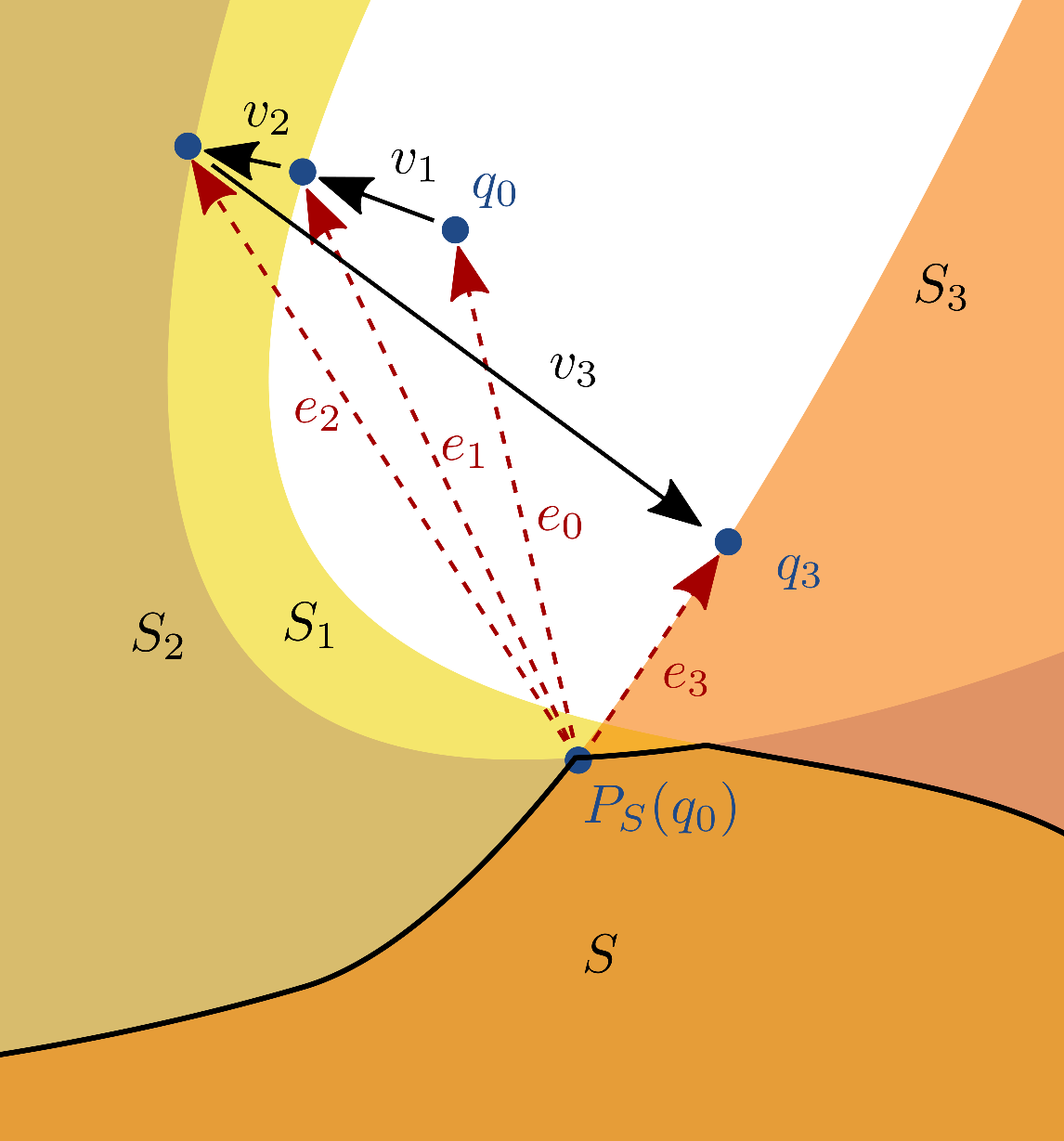}
    \caption{Example for the importance of the metric calmness of the intersection. If $d_{S_j}(q_{j-1}) = \Vert v_j \Vert$ would always be too small compared to $d_S(q_{j-1}) = \Vert e_{j-1} \Vert$, then the error could potentially increase with every projection. Metric calmness is exactly the property which bounds this ratio from below.}
    \label{fig:projection_error_proof}
\end{figure}

\begin{theorem}[Iterated projection error]
\label{lem:iter-proj_error}
    Let $S \coloneqq \bigcap_{j=1}^M S_j$ be such that \cref{ass:S_prox-reg,ass:Sj_prox-reg,ass:S_Sj,ass:metric_calm} holds. 

    Then, there exists an $\varepsilon_P > 0$ and $L_P \in (0,1)$ such that for any 
    \begin{align} 
    q_0 \in S + B_{\varepsilon_P},
    \end{align}
    we have that
    \begin{align} \label{eq:estimate_iteration}
        \Vert P^{\mathrm{itr}}_{S}( q_0 ) - P_S(q_0) \Vert \leq L_P \, d_S( q_0 ) = L_P \Vert q_0 - P_S(q_0) \Vert.
    \end{align}
\end{theorem}

\begin{proof}
First, we note that the statement is trivially true if $\Vert e_m \Vert = 0$. Hence, we can assume that $\Vert e_j \Vert \neq 0$ for all $1 \leq j \leq m$ since otherwise $\Vert e_m \Vert = 0$.

We define $\varepsilon_\eta = 2^{-m}\eta$, consider an arbitrary $\varepsilon ~\in ~(0,\varepsilon_\eta)$, and select an arbitrary point $q_0 ~\in~ S + B_{\varepsilon}$. By applying the iterative projection, we obtain $q_j, v_j$ and $e_j$ as in \cref{eq:def:qj,eq:def:vj,eq:def:ej}.

The condition $\varepsilon < \varepsilon_\eta$ and \cref{lem:trivial-proj-bound} ensure that all projections $q_j = P_{S_j}(q_{j-1})$ are unique, and we have
\begin{align}
    \Vert v_j \Vert \leq 2^{j-1} d_S(q_0) < 2^{j-1}\varepsilon< 2^{j-1}\varepsilon_\eta = 2^{j-1} 2^{-m} \eta \leq \frac{1}{2} \eta 
    \quad \text{for all } 1 \leq j \leq m,
    \label{eq:proof_itr:vj}
\end{align}
which makes \cref{lem:single_proj} applicable. Hence, we have that
\begin{align} \label{eq:boundej}
    \Vert e_j \Vert^2 \leq f_j^2 \Vert e_{j-1} \Vert^2 \leq \left( \prod_{k=1}^{j} f_k^2 \right) \Vert e_0 \Vert^2
    \quad 
    \text{for all } 1 \leq j \leq m,
\end{align}
where 
\begin{align}
        f_j \coloneqq \left(1 - \frac{\Vert v_{j} \Vert}{\eta}\right)^{-\frac{1}{2}} \left( 1 - \frac{\Vert v_j \Vert^2}{\Vert e_{j-1} \Vert^2} \right)^{\frac{1}{2}}.
        \label{eq:proof:fj}
    \end{align}
 In particular, it holds that
\begin{align}
    \Vert e_m \Vert  \leq \left( \prod_{k=1}^{m} f_k \right) \Vert e_0 \Vert,
\end{align}
which is equivalent to the expression \eqref{eq:estimate_iteration} that we want to prove with the constant $L_P$ given by: 
$$L_P \coloneqq \sup_{q_0 \in S + B_{\varepsilon_P}} \,\, \prod_{k=1}^{m} f_k.$$
So we are just left with proving that there exists $\varepsilon_P \in (0,\varepsilon_\eta)$ such that 
$L_P\in (0,1)$.

\medskip
To achieve this, we first obtain a rough estimate of the coefficients $f_j$. From expression \cref{eq:proof_itr:vj} we have that
\begin{align}
    \Vert v_j \Vert < 2^{m} \varepsilon
    \quad 
    \text{for all } 1 \leq j \leq m,
\end{align}
on the other hand, it holds that (recall remark \ref{rem:bounds})
\begin{align} \label{eq:rough_estimate}
    1-\frac{\Vert v_j\Vert^2}{\Vert e_{j-1}\Vert^2}\in (0,1].
\end{align}
Combining these two bounds we have that
\begin{align}
    f_j < \left( 1 - 2^m \frac{\varepsilon}{\eta} 
    \right)^{-\frac{1}{2}} 
    \eqqcolon 
     \overline{f}(\varepsilon),
     \label{eq:proof:feps}
\end{align}
(notice that $\overline f(\varepsilon)$ is real since $\varepsilon < 2^{-m} \eta$).
With this we get the rough estimate
\begin{align}\label{eq:rough_estimate_Lp}
  L_P \leq \overline f(\varepsilon)^m.  
\end{align}
Now, since  $\overline f(\varepsilon)\in [1,\infty)$ and it decreases towards 1 as $\varepsilon \to 0$, we can define $\delta_\varepsilon$ as
$$\overline f(\varepsilon)^m = 1+\delta_\varepsilon,$$
fulfilling $\delta_\varepsilon\geq 0$ decreasing towards 0 as $\varepsilon\to 0$.

Since $\overline f(\varepsilon)^m\geq 1$, the estimate \eqref{eq:rough_estimate_Lp}
 is not good enough to prove the result. To refine this estimate, we will show that, for $\varepsilon$ small enough, there exists at least one of the coefficients $f_j$ that becomes very small, making $L_P<1$. We show this next.

First, consider $\varepsilon_P\in (0, \varepsilon_\eta)$ small enough such that the interval
$$I_{\varepsilon_P}:= \left(\sqrt{1-(1+\delta_{\varepsilon_P})^{-2}}, \frac{\alpha}{m(1+\delta_{\varepsilon_P})} \right) \subset \mathbb{R}$$
is non-empty. (We remind that the constant $\alpha$ appears in assumption \cref{ass:metric_calm}). This can be achieved for $\varepsilon_P$ small enough since $\delta_{\varepsilon_P} \to 0$ as $\varepsilon_P\to 0$.
Now we pick a value $\beta=\beta(\varepsilon_P)$ within this interval, i.e.,
\begin{align} \label{eq:definition_beta}
    \beta_{\varepsilon_P}\in I_{\varepsilon_P}.
\end{align}
In particular, one can check that $\beta_{\varepsilon_P}$ satisfies the following two bounds:
\begin{align}\label{eq:proof:alpha_less_m_beta}
    \overline{f}(\varepsilon_P)^{m} \sqrt{1 - \beta_{\varepsilon_P}^2} < 1 
    \qquad \text{and} \qquad
    \beta_{\varepsilon_P} m \overline{f}(\varepsilon_P)^{m} < \alpha.
\end{align}

\medskip
Next, we will prove that 
\begin{align}
    \Vert v_{j^*} \Vert \geq \beta_{\varepsilon_P} \Vert e_{j^* - 1} \Vert, \quad \mbox{for some index } j^*\in\{1,\hdots,m\}.
    \label{eq:proof:case1}
\end{align}
If this holds true, then we can improve the estimate \eqref{eq:rough_estimate} since then
$$1-\frac{\Vert v_j\Vert^2}{\Vert e_{j-1}\Vert^2}\leq 1-\beta_{\varepsilon_P}^2,$$
which implies that
\begin{align}
    f_{j^*} < \left( 1 - 2^m \frac{\varepsilon}{\eta} 
    \right)^{\frac{-1}{2}} \left( 1 - \beta_{\varepsilon_P}^2 \right)^{\frac{1}{2}}
    \leq 
    \overline{f}(\varepsilon) \sqrt{ 1 - \beta_{\varepsilon_P}^2 }.
\end{align}
With this estimate for $f_{j^*}$ and estimate \eqref{eq:proof:feps} for $f_j$ with $j\neq j_*$ we have that 
\begin{align}
    L_P\leq \overline{f}(\varepsilon)^{m} \sqrt{ 1 - \beta_{\varepsilon_P}^2 }.
    \label{eq:proof:case1_claim}
\end{align}
Thanks to the first bound in \eqref{eq:proof:alpha_less_m_beta}, we conclude that $L_P\in (0,1)$, which concludes the result.

So, we are left with showing that indeed \eqref{eq:proof:case1} holds true. We show this by contradiction. Suppose that the opposite holds, i.e.,  
\begin{align}
    \Vert v_j \Vert < \beta_{\varepsilon_P} \Vert e_{j-1} \Vert \quad \mbox{for all }1 \leq j \leq m.
    \label{eq:proof:case2}
\end{align}

By the assumption of a metrically calm intersection \cref{ass:metric_calm},  there exists an index $\ell$ such that
\begin{align}
    \alpha d_S(q_0) \leq  d_{S_\ell}(q_0). \label{eq:case2_first_estimate}
\end{align}
We use $q_0 = q_{\ell-1} -\sum_{j=1}^{\ell-1} v_j$ and $d_{S_\ell}(q_{\ell-1}) = \Vert v_{\ell} \Vert$ (see \cref{fig:projection_error_proof}) to compute 
\begin{align}
    d_{S_\ell}(q_0) 
    &\leq d_{S_{\ell}}(q_{\ell-1}) + \sum_{j=1}^{\ell-1} \Vert v_j \Vert 
    \\
    &= \sum_{j=1}^{\ell} \Vert v_j \Vert \\ 
    &< \beta_{\varepsilon_P} \sum_{j=0}^{\ell-1} \Vert e_j \Vert \\ 
    &< \beta_{\varepsilon_P} \sum_{j=0}^{\ell-1} \overline{f}(\varepsilon)^j \Vert e_0 \Vert \\
    &\leq \beta_{\varepsilon_P} m \overline{f}(\varepsilon)^{m} \Vert e_0 \Vert
= \beta_{\varepsilon_P} m \overline{f}(\varepsilon)^{m} d_S(q_0),
    \label{eq:case2_second_estimate}
\end{align}
where in the third inequality we used \eqref{eq:proof:case2}, and in the fourth we combined \eqref{eq:boundej}  and \eqref{eq:proof:feps}.
Combining \cref{eq:case2_first_estimate,eq:case2_second_estimate}, we get the bound
\begin{align}
    \alpha < \beta m \overline{f}(\varepsilon)^m.
    \label{eq:proof:case2_alpha_greater_m_beta}
\end{align}
But this cannot be true since it violates the second bound of \eqref{eq:proof:alpha_less_m_beta}. We have reached a contradiction which shows that, indeed, \eqref{eq:proof:case1} holds.

\end{proof}

\subsection{Proof of the numerical consistency and stability}

For the proof of consistency, we need to extend \cref{prop:subdiff_susc} slightly to hold for sums of subdifferentials. 
\begin{corollary}
\label{cor:subdiff_sum_susc}
Let $S = \bigcap_{j=1}^m S_j$ be such that \cref{ass:S_prox-reg,ass:S_Sj} holds. Moreover, let $x \in S$ and $(q_j^{(n)})_{n \in \mathbb{N}}$ be sequences with $\lim_{n \to \infty} q_j^{(n)} = x$ and $q_j^{(n)} \in S_j$ for all $1 \leq j \leq m$.
Then, 
\begin{align}
    \limsup_{n \to \infty} \sigma\left( -\sum_{j=1}^m \partial^P d_{S_j}(q_j^{(n)}), w \right) \leq m \sigma( -\partial^P d_S(x), w) 
\end{align}
for all $w \in \mathbb{R}^d$. 
\end{corollary}

\begin{proof}
Let $w \in \mathbb{R}^d$ be arbitrary. By \cref{prop:subdiff_susc} we have for each $1 \leq j \leq m$ the relation
\begin{align}
    \limsup_{n\to\infty} \sigma( -\partial^P d_{S_j}(q_j^{(n)}), w ) \leq \sigma( -\partial^P d_{S_j}(x), w ).
\end{align}
Since $A \mapsto \sigma(A,w)$ is Minkowsky additive and $\limsup$ is subadditive, we get 
\begin{align}
     \limsup_{n \to \infty} \sigma\left( -\sum_{j=1}^m \partial^P d_{S_j}(q_j^{(n)}), w \right)&=
     \limsup_{n \to \infty}  \sum_{j=1}^m  \sigma\left( -\partial^P d_{S_j}(q_j^{(n)}), w \right) \\
     &\leq 
      \sum_{j=1}^m \limsup_{n \to \infty} \sigma\left( -\partial^P d_{S_j}(q_j^{(n)}), w \right) \\
      &\leq 
      \sum_{j=1}^m \sigma\left( -\partial^P d_{S_j}(x), w \right)
      \\
      &=  
      \sigma\left( -\sum_{j=1}^m \partial^P d_{S_j}(x), w \right)
      \\
      &\leq
      m \sigma( -\partial^P d_S(x), w ),
\end{align}
where we used \cref{thm:nc_intersection} in the last step, i.e., 
$\sum \partial^P d_{S_j}(x) \subseteq m \partial^P d_S(x)$ and \eqref{eq:monotonicity_support_function}.
\end{proof}

We are ready to prove the convergence of the PBD method. The next theorem shows that \cref{thm:abstract_convergence} can be applied to the PBD method.

\begin{theorem}[Consistency and stability for PBD]
\label{th:consistency_stability_PBD}
Assuming \cref{ass:f:lipschitz,ass:f:bounded,ass:S_prox-reg,ass:S_Sj,ass:Sj_prox-reg,ass:metric_calm}  hold true, then the PBD method 
\begin{align}
    \Phi_h^{\mathrm{PBD}}(x) \coloneqq P_{S_m} \circ \cdots \circ P_{S_1}( x + h f(x) )
\end{align}
 is consistent and stable in the sense of \cref{def:consistency} and \cref{def:stability}.

Explicitly stated, for $T > 0$ and an initial condition $x_0 \in S$, 
there exists $\varepsilon_0, A, C, K > 0$ such that:
\begin{itemize}
    \item For sufficiently small $h > 0$, it holds
\begin{align}
    d_S( \Phi_h^{\mathrm{PBD}}(x) ) &< K h
    \quad 
    \text{for all } x \in S + B_{K h}.
    \label{eq:main_thm:fesible}
\end{align} 
\item For sufficiently small $h > 0$, any $1 \leq k \leq \frac{T}{h}$ it holds
\begin{align}     
    \Vert \Phi_h^{\mathrm{PBD}}(x_k) - x_k \Vert &\leq h A 
    \label{eq:main_thm:stable}
\end{align}
where $x_k = \left( \Phi_h^{\mathrm{PBD}} \right)^k( x_0 )$.
\item For any $x^* \in S$, we set $h^{(n)} = \frac{T}{n}$ and consider an arbitrary sequence $y^{(n)} \to x^*$ with $d_S(y^{(n)})~ <~ K h^{(n)}$. Then, for all $v \in \mathbb{R}^d$, it holds that
\begin{align}
    \limsup_{n \to \infty} 
    \sigma
    \left( 
    \left\{ 
    \frac{1}{h^{(n)}} 
    \left(
    \Phi^{\mathrm{PBD}}_{h^{(n)}}(y^{(n)}) - y^{(n)} 
    \right)
    - f(y^{(n)})
    \right\},
    \, v
    \right) 
    \leq 
    \sigma(
    -C \partial^P d_S(x^*),
    \, v
    ).
    \label{eq:main_thm:consistent}
\end{align}
\end{itemize}
\end{theorem}

\begin{proof}
\textbf{Step 1.} We start by showing \cref{eq:main_thm:fesible}.

Let $\varepsilon_P$ denote the constant from \cref{lem:iter-proj_error} such that $\Vert P^{\mathrm{itr}}_S(x) - P_S(x) \Vert \leq L_P \, d_S(x)$ holds for all $x \in S+ B_{\varepsilon_P}$. 
We then define
\begin{align}
    h^{(n)} \coloneqq \frac{T}{n} \quad \text{for } n \in \mathbb{N}
\end{align}
and we choose a value $K \in \mathbb{N}$ such that 
\begin{align}
    L_P ( K + M_f ) < K,
    \label{eq:proof:pbd_stab:cond_K}
\end{align}
(recall that $M_f$ is given in assumption \ref{ass:f:bounded}).
This is possible since $\lim_{K \to \infty} L_P \frac{( K + M_f )}{K} = L_P < 1$.  
Next, we choose an $n_0 \in \mathbb{N}$ such that
\begin{align}
    (K + M_f) h^{(n_0)} &< \varepsilon_P \label{eq:proof:pbd_stab:cond_n_0}.
\end{align}
For any $h \in (0, h^{(n_0)})$ and any $x \in S + B_{Kh}$, we can compute with \cref{eq:proof:pbd_stab:cond_K,eq:proof:pbd_stab:cond_n_0}  and \cref{ass:f:bounded} that
\begin{align}
    d_S( x + h f(x) ) \leq d_S(x) + h \Vert f(x) \Vert < K h + M_f h < \varepsilon_P.
\end{align}
Hence, \cref{lem:iter-proj_error} is applicable to $q \coloneqq x + h f(x)$, which yields 
\begin{align}
d_S( \Phi_h^{\mathrm{PBD}}(x) ) &= 
    d_S( P^{\mathrm{itr}}_S(q) ) \\
    &= \Vert P_S(q) - P^{\mathrm{itr}}_S(q) \Vert  
    \\
    &\leq L_P \, d_S(q) \\ 
    &< L_P \, ( K h + M_f h ) \\ 
    &< K h^{(n)},
    \label{eq:main_proof:dS_eps}
\end{align}
where we applied \cref{eq:proof:pbd_stab:cond_K} in the last step. This proves \cref{eq:main_thm:fesible}.

\textbf{Step 2.} To show \cref{eq:main_thm:stable}, we select again $h \in (0, h^{(n_0)})$, $x_0 \in S$ and any index $1 \leq k \leq \frac{T}{h}$, we define $x_k = (\Phi_h^{\mathrm{PBD}})^k (x_0)$.

Since $x_0 \in S$, we can apply \cref{eq:main_proof:dS_eps} and, therefore, $d_S(\Phi_h^{\mathrm{PBD}}(x_0)) < Kh$ and by iteration
\begin{align}
    d_S( \Phi_h^{\mathrm{PBD}}( x_k ) ) < \varepsilon = K h.
\end{align}

Now, we set $q \coloneqq x_k + h f(x_k)$ and compute
\begin{align}
    \Vert \Phi_h^{\text{PBD}}(x_k) - x_k \Vert 
    &\leq
    \Vert P^{itr}_S(  q ) - P_S( q ) \Vert 
    + \Vert P_S( q ) -  q \Vert 
    + \Vert  q - x_k \Vert \\ 
    &\leq L_P \, d_S( q ) + d_S( q ) + h M_f \\
    &\leq (1+L_P) \, d_S(q) + h M_f \\
    &\leq (1 + L_P) ( h K + h M_f ) + h M_f \\ 
    &\leq h \left( (1 + L_P) (K + M_f) + M_f \right).
\end{align}
This proves the stability \cref{eq:main_thm:stable} with $A = (1+L_P)(K + M_f) + M_f$.

\textbf{Step 3.} To show \cref{eq:main_thm:consistent}, we fix $x^* \in S$, define $h^{(n)} = \frac{T}{n}$ and we consider an arbitrary sequence $y^{(n)} \to x^*$ with $d_S(y^{(n)}) < K h^{(n)}$ .

First, we denote $q_0^{(n)} \coloneqq y^{(n)} + h^{(n)} f(y^{(n)})$ and we define $q_{j}^{(n)} \coloneqq P_{S_j}( q_{j-1}^{(n)} )$ for $1 \leq j \leq m$.
We have $d_S(q_0^{(n)}) \leq d_S(y^{(n)}) + h^{(n)} M_f < (K + M_f) \, h^{(n)}$, which makes \cref{lem:trivial-proj-bound} applicable, for large enough $n$. In particular, the projections $P_{S_j}( q_{j-1}^{(n)} )$ are single-valued for large enough $n$.

By definition of the proximal normal cones, we get 
\begin{align}
    q^{(n)}_{j-1} - q^{(n)}_{j} \in N(S_j, q^{(n)}_{j} ),
\end{align}
(recall remark \ref{rem:errors_and_v}).
Moreover, \cref{lem:trivial-proj-bound} implies that for all $1 \leq j \leq m$
\begin{align}
    \Vert q_{j-1}^{(n)} - q_{j}^{(n)} \Vert &\leq 2^m d_S(q^{(n)}_0) < \tilde{C} h^{(n)},
    \quad 
    \text{for large enough } n,
\end{align}
where $\tilde{C} = 2^m (K + M_f) $,
which implies that
$$q_{j}^{(n)} \to x^* \quad \text{as } n \to \infty,$$
since $y^{(n)} \to x^*$ as $n\to \infty$.

This implies
\begin{align}
    q^{(n)}_{j-1} - q^{(n)}_{j} \in \tilde{C} h^{(n)} \, \partial^P d_{S_j}( q_{j}^{(n)} ).
\end{align}

Combining all these terms leads to 
\begin{align}
    \frac{\Phi_h^{\text{PBD}}(y^{(n)}) - y^{(n)}}{h^{(n)}} - f(y^{(n)}) 
    &= \sum_{j=1}^m \frac{ q_{j}^{(n)} - q_{j-1}^{(n)} }{h^{(n)}} 
    \\
    &\in -\tilde{C} \sum_{j=1}^m \partial^P d_{S_j}( q_{j}^{(n)} ).
\end{align}
Applying \cref{cor:subdiff_sum_susc} yields \cref{eq:main_thm:consistent} with $C = m \tilde{C}$.
\end{proof}

\begin{theorem}[Convergence of PBD]
\label{thm:pbd}
    Suppose that assumptions  \cref{ass:f:lipschitz,ass:f:bounded,ass:S_prox-reg,ass:S_Sj,ass:Sj_prox-reg,ass:metric_calm} hold (given in section \ref{sec:assumptions}), then position-based dynamics $\Phi_h^{PBD}$ is a convergent method, i.e., the numerical trajectory $x^{(n)}~:~[0,T]~\to ~\mathbb{R}^d$ computed with $\Phi_h^{\text{PBD}}$ for step size $h = \frac{T}{n}$ satisfies
\begin{align}
    \sup_{t\in[0,T]}\Vert x^{(n)}(t) - x(t) \Vert \to 0 
    \quad 
    \text{as } n \to \infty,
\end{align}
where $x^{(n)}$ is the linear interpolation defined in \eqref{eq:num_traj} and
where $x$  denotes the  solution of the differential inclusion \eqref{eq:differential_inclusion}. 
\end{theorem}

With this result, we are finally ready to conclude the main result of this paper:
\begin{proof}[Proof of \ref{thm:pbd}]
    By theorem \ref{th:consistency_stability_PBD} we know that the PBD method is stable and consistent, so we can apply theorem \ref{thm:abstract_convergence} and the result follows.
\end{proof}
 
\section{Numerical convergence tests for PBD}
\label{sec:num_tests}

We demonstrate the convergence of the first-order PBD method with two numerical tests, and we compare the efficiency of the PBD method with the related projected nonlinear Gauss-Seidel method (PNGS), the projected Gauss-Seidel method (PGS) and simple penalizing-constraint schemes. We note that our comparison is not exhaustive, as a detailed computational comparison of PBD with other numerical methods is beyond the scope of this article. There are known acceleration techniques for the PBD method \cite{Wang2015,Francu2017}, and in general, the efficiency depends highly on the application \cite{Chen2023}. We summarise the above-mentioned numerical methods in the following section, but the reader can find extensive explanations on the numerical methods for differential equations in \cite{Acary2008} and their implementations in \cite{Acary2023}.

\subsection{Numerical methods}

Penalizing-constraint schemes are among the only explicit time-stepping schemes for differential inclusions. For the numerical comparison, we consider the scheme, which is based on the ODE
\begin{align}
    \dot{x} = 
    f(x) - \gamma \sum_{j=1}^m d_S(x) \nabla d_S(x),
\end{align}
where $\gamma$ is a numerical parameter to determine the strength of the penalty terms and $\nabla$ denotes the gradient of the distance, which is single-valued whenever $d_S(x) \neq 0$. For the comparison, we use the explicit Euler method, which is not optimal but allows us to focus solely on how solvers deal with the constraints. We denote the resulting scheme as the penalty method with parameter $\gamma$.

For our examples, the projection $P_{S_j}$ will always be explicitly known. In this setting, the PNGS-based time-stepping reads
\begin{align}
    \Phi_h^{\mathrm{PNGS}}(x) = \left( P_S^{\mathrm{itr}} \right)^{k^*}(x + h f(x)),
\end{align}
where $k^*$ in each time step is the first index satisfying the stopping criterium
\[
\Vert P^{\mathrm{itr}}_S(y) - y \Vert 
\leq \mathrm{abstol} + \Vert y \Vert \mathrm{reltol}
\quad 
\text{for }
y = \left( P^{\mathrm{itr}}_S\right)^{k^*}(x)
\label{eq:stopping_crit}
\]
where $\mathrm{abstol}, \mathrm{reltol}$ are the tolerance parameters. The convergence of this method is shown, for example, in \cite{Wu2020}.

One obtains the equivalent of the projected Gauss-Seidel methods by replacing $S$ with the set 
\begin{align}
\tilde{S}(x) \coloneqq 
\{ y \in \mathbb{R} \mid \exists v \in \partial^P d_S(x), 0 = d_S(x) + \langle v, y - x\rangle \}.
\end{align}
The resulting scheme reads
\begin{align}
     \Phi_h^{\mathrm{PNGS}}(x) = \left( P_{\tilde{S}(x + h f(x))}^{\mathrm{itr}} \right)^{k^*}(x + h f(x))
\end{align}
where $k^*$ is again determined by the stopping criterium \cref{eq:stopping_crit}.
We note that the convergence of the method $x_{k+1} = P_{\tilde{S}(x_k)}(x_k + h f(x_k))$ was studied in \cite{Venel2011}, which essentially contains the proof of \cref{thm:abstract_convergence}.

\subsection{Sliding case}
\label{sec:numtest:sliding}

The first test case is a constructed academic example with explicit solutions. Intuitively, the following system describes the motion of an overdamped particle that slides along the intersection of $N-1$ many spheres.

We construct this trajectory by considering $N-1$ spheres, which intersect exactly in a way such that the trajectory is initially following the curve
\begin{align}
    \{ x \in \mathbb{R}^d \mid x_1^2 + x_2^2 = 1^2, x_j = 0 \text{ for all } 3 \leq j \leq d \}.
\end{align}
For given dimension $d \geq 2$ and a constant $C > 0$, we define 
$R = \sqrt{1^2 + C^2}$ and
\begin{align}
    S_j &\coloneqq B_R(C e_{j+2}), 
    \quad 
    \text{for } 1 \leq j \leq d - 2 \\
    S_{d-1} &\coloneqq B_R\left( \sum_{j=3}^d C (d-2)^{-\frac{1}{2}} e_j\right),
\end{align}
where $(e_j)_{1 \leq j \leq d}$ are the canonical basis vectors of $\mathbb{R}^d$,  and we denote $S = \bigcap_j S_j$. By construction, the set $S$ is $\eta$-prox-regular with $\eta = \min(R,1)$.

The constant $C$ determines the distance of the spheres from the center, for small values of $C$ the spheres are very close, and the created valley at the intersection of the spheres is rather shallow. For large values of $C$, the created valley is very deep and promoted, which is numerically more challenging.

We consider the system
\begin{align}
    \dot x &= -e_2 - N(S,x),\\ 
    x(0) &= \sin(\alpha) e_1 + \cos(\alpha) e_2 ,
\end{align}
where $\alpha \in (0, \pi)$ determines the initial position.

This system has the following explicit solution for the sliding phase:
\begin{align}
    x_1(t) = \frac{C_2 e^t}{1 + e^{2(c_1+t)}}, \quad 
    x_2(t) = \frac{1 - e^{2(c_1+t)}}{1 + e^{2(c_1+t)}},\quad 
    x_j(t) = 0 \quad \text{for } 3 \leq j \leq d.
\end{align}
with the constants $C_1 = \frac{1}{2} \log( \frac{1 - \cos(\alpha)}{1 + \cos(\alpha)} )$ and $C_2 = \sin(\alpha) (e^{2c_1} + 1)$.
After the sliding phase, for $t > C_1$, the trajectory leaves the boundary of $S$ and becomes 
\begin{align}
    x_1(t) = 1, \quad x_2(t) = C_1 - t,
    \quad 
    x_j(t) = 0 \quad \text{for } 3 \leq j \leq d.
\end{align}

The numerical comparison is presented in \cref{fig:conv_example_valley}. Our analysis confirms the convergence of PBD for this example, and the observed order of convergence is $1$. 

Comparison between PBD and the penalty method shows that even with close to optimal penalty parameters, the PBD method outperforms penalty methods in terms of numerical error at a fixed time-step (left plot in \cref{fig:conv_example_valley}). Since both methods have similar computational cost, the PBD method is also more accurate when fixing the computational budget (right plot in \cref{fig:conv_example_valley}).

For large time stepsizes, the numerical error of the PBD method remains bounded, which is not the case for the penalty method or the PGS method, as in these methods, instability issues become accentuated. The stability of PBD at large time-steps is the main reason for its popularity in computer graphics, where instability would lead to unwanted visual artifacts.

Finally, if we compare PBD with the PNGS and PGS methods, we observe that these methods are more accurate when comparing at same time stepsizes. However, the precision-work diagram (right plot in \cref{fig:conv_example_valley}) shows that the PBD has a smaller error per computational budget.

\begin{figure}[h]
    \centering
    \includegraphics[width=\textwidth]{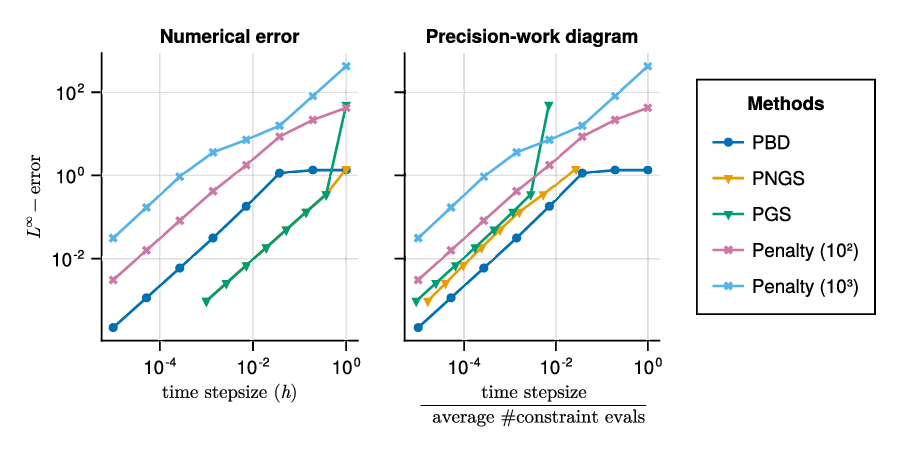}   
    \caption{\textbf{Left:} Comparison of the numerical error 
    for the test problem with parameters $C = 10, d = 3$ and $\alpha = \frac{\pi}{16}$. This shows in particular convergences of PBD for this test case. 
    \textbf{Right:} The precision-work diagram takes into account the computational budget by rescaling the $x$-axis by the average number of constraint evaluations and keeping the $y$-axis as before. This serves as a proxy for the computational runtime, with smaller values representing longer runtimes. Notably, the PBD method outperforms all other methods in terms of accuracy per computational budget, also for large timesteps.}
    \label{fig:conv_example_valley}
\end{figure}

\subsection{Non-overlapping disks}
\label{sec:numtest:nonoverlap}

The second test case considers $N$ disks that are attracted to a center and constrained within a half-space. A similar setting is considered in \cite{Venel2011,Pecol2011} with the aim of modelling crowd movement. 

This test case reflects a very common application for PBD. In NVIDIA's particle engine Flex which is part of PhysX  \cite{Macklin2014,NVIDIA2018}, non-overlapping spheres are the core of rigid body simulations and fluid dynamics, where each of these continua is represented as a collection of non-overlapping spheres with additional constraints. 
 
Our test model is defined as follows. For a given radius $R > 0$, we define the feasible sets as 
\begin{align}
    S_{ij} = \{ x \in (\mathbb{R}^2)^{N} \mid 
    \Vert x_i - x_j \Vert \geq 2R\} \quad \text{and} \quad
    S \coloneqq \bigcap_{i = 1}^N \bigcap_{j=1}^{i-1} S_{ij}.
\end{align}
The sets represent all center positions for non-overlapping disks with radius $R$. We note that the projection $P_{S_{ij}}$ is trivial to compute, whereas $P_S$ is a difficult non-convex optimization problem.

We consider the differential inclusion 
\begin{align}
    \dot{x} = -\gamma x - N(S,x),
\end{align}
where $\gamma$ determines the strength of the attraction to the center, and as an initial condition, we pick some arbitrary $x_0 \in S$.
See \cref{fig:spheres_init} for an example of the trajectories.

\begin{figure}[h]
    \centering
    \includegraphics[width=.48\linewidth]{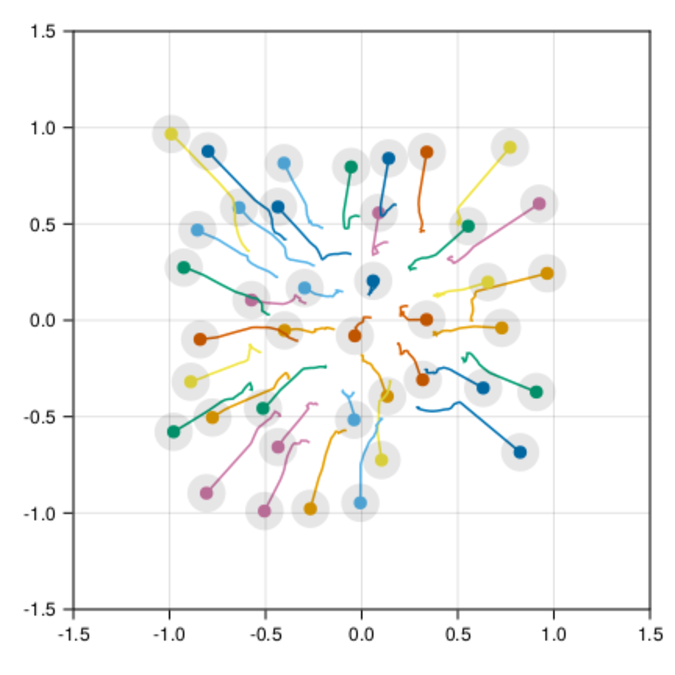}
    \hfill
    \includegraphics[width=.48\linewidth]{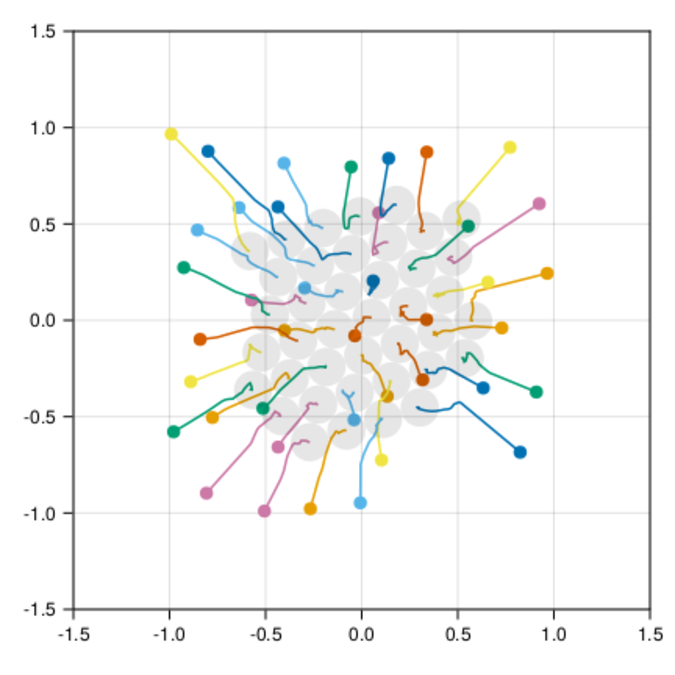}
    \caption{Initial condition (left) and terminal condition (right) of the non-overlapping disks test model for parameters $R = 0.1$, $N = 40$, $T = 4.0$, and $\gamma = 1.0$.}
    \label{fig:spheres_init}
\end{figure}

The reference solution was computed with the PNGS method with 
the parameters $\texttt{reltol} = 10^{-10}, \texttt{abstol} = 10^{-12}$ and time stepsize $h = 10^{-6}$.

We present the numerical convergence plot and precision-work diagram in \cref{fig:conv_example_spheres}.

The convergence analysis shows a similar picture as in \cref{sec:numtest:sliding}. Again, PBD performs best in the precision-work diagram, showing the smallest errors at fixed computational budgets. 

This test case displays the typical behavior of penalty methods, which is that the choice of the penalty parameter ($\gamma$) impacts the numerical error. In the case of non-overlap constraints, too small penalty parameters lead to a breakdown of the convergence already at low accuracy. In contrast, larger penalty parameters 
can sustain the convergence longer but require smaller time steps until convergence. Independent of the time stepsize and the penalty parameter, PBD shows a smaller numerical error than the penalty method.

\begin{figure}[h]
    \centering
    \includegraphics[width=\textwidth]{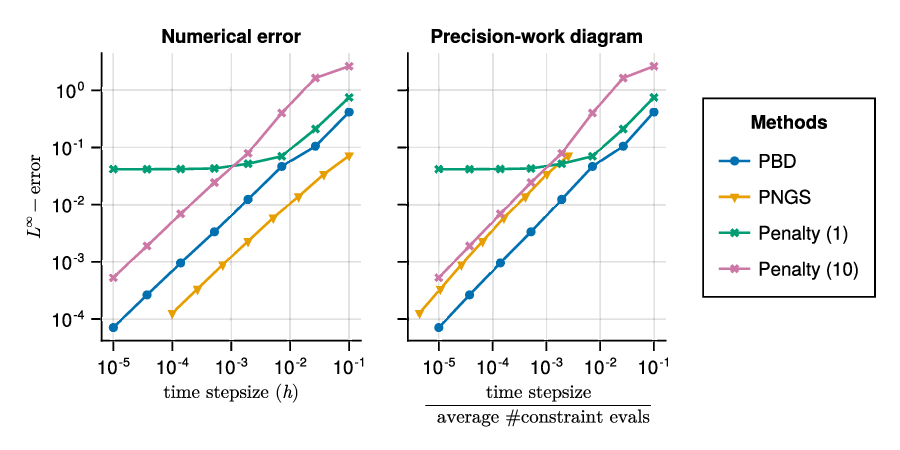}   
    \caption{
    Overall, the convergence behavior is similar to \cref{fig:conv_example_valley}. 
    \textbf{Left:}  This experiment shows again that PNGS is more accurate when one compares errors per time stepsizes. Notably, we see that convergence of the penalty methods depends on the penalty parameter and eventually breaks down, whereas PBD does not require such a numerical parameter for convergence.
    \textbf{Right:} The PBD method is again best in terms of performance per runtime in this example. }
    \label{fig:conv_example_spheres}
\end{figure}

\section{Conclusion}
\label{sec:conclusion}

This article shows, for the first time, that the fusing of force integration and constraint projection, as 
done in the position-based dynamics method, can lead to convergent methods in a rigorous mathematical sense. 
Our result is restricted to first-order dynamics, which is still applicable in many biological models or constrained gradient flow problems. 

Our proof represents an initial step for further development of numerical analysis for the PBD method. The most pressing question is how the second-order case could be treated. The analysis of second-order differential inclusions is considerably more challenging, and aspects such as impact laws lead to systems with discontinuous velocities, which require a different solution theory (using measure differential inclusions).

Also, for the first-order case, there are various next steps to investigate. For example, numerical experiments suggest that the order of convergence is $1$ but a mathematical proof is still missing. Another direction would be to allow the feasible set $S$ to change over time. The theory of differential inclusions for such so-called sweeping processes is already well-developed, and it is likely that the techniques presented in this article generalise to non-stationary feasible sets.
\begin{appendices}
\crefalias{section}{appendix}

\section*{Acknowledgement}

This research was partially supported by the Austrian Science Fund (FWF) through the  project  F65. The work of SP and SMA was funded  by the Vienna Science  and  Technology  Fund  (WWTF)  [10.47379/VRG17014].

\section{Proof of the convergence for numerical flows 
        \ifpdf\else(\cref{thm:abstract_convergence})\fi}
\label{app:conv_proof}
The following proof is a slight adaptation of the main proof in \cite{Bernicot2010,Venel2011} mixed with the notion of scalar upper semicontinuity as in \cite{Edmond2006,Adly2017}.

\begin{proof}[Proof of \cref{thm:abstract_convergence}]
In the following, we will use the notation $h = h^{(n)} = \frac{T}{n}$ for $n \in \mathbb{N}$ and $t_k = t_k^{(n)} = k h^{(n)}$.
Let $\Phi_h : \mathbb{R}^d \to \mathbb{R}^d$ be a numerical flow under the assumptions of the theorem.

The numerical approximation of system \cref{eq:def:main_di,eq:def:main_di_init} is given by
\begin{align}
    x_{k+1}^{(n)} \coloneqq \Phi_h(x_{k}^{(n)}) \quad \text{for } k > 0.
\end{align}
We denote the piece-wise linear trajectory as
\begin{align}
    x^{(n)} : [0,T] \to \mathbb{R}^d
\end{align}
as defined in \eqref{eq:num_traj}.
In the following, we  will refer to the start and endpoints of each integration step, therefore, we define for $t \in [0,T]$ 
\begin{align}
    \lfloor t \rfloor^{(n)} = t_k^{(n)}, \quad 
    \lceil t \rceil^{(n)} = t^{(n)}_{k+1} \quad \text{if } t \in [t_k^{(n)}, t_{k+1}^{(n)}).
\end{align}
For each $n$, we also define the corresponding forces used by the numerical method as a piece-wise constant function
\begin{align}
    f^{(n)}(t) \coloneqq f(x^{(n)}( \lfloor t \rfloor^{(n)} ) ).
\end{align}
We note that $f^{(n)} : [0,T] \to \mathbb{R}^d$ is a function in $t$ which depends on the numerical solution $x^{(n)}$.

1. The stability of the numerical flow implies for sufficiently large $n$ that
\begin{align}
\Vert \dot x^{(n)}(t) \Vert \leq h^{(n)} ( A + B \Vert x^{(n)}(t)  \Vert)
\quad \text{for all } t \in [0,T].
\label{eq:conv_proof:stab}
\end{align}
A discrete variant of the Gronwall lemma implies that 
\begin{align}
    \Vert x^{(n)} \Vert_{\infty} \leq C_1 \coloneqq \exp(BT) + A \frac{\exp{BT} - 1}{\exp{BT}} ,
\end{align}
and inserting $C_1$ back into \cref{eq:conv_proof:stab} implies
\begin{align}
    \Vert \dot{x}^{(n)} \Vert_{\infty} \leq C_2 \coloneqq A + B C_1.
\end{align}
(We recall that $\Vert \cdot \Vert_\infty$ denotes the sup norm.)

2. By \cite[Excersice 8.2]{Brezis2011} each bounded sequence in $W^{1,\infty}([0,T], \mathbb{R}^{d})$ 
has a subsequence (which we denote without relabeling) which has a limit $x \in L^{\infty}([0,R],\mathbb{R}^d)$ such that
\begin{align}
    x^{(n)} &\to x  \quad \text{in } L^{\infty}([0,T],\mathbb{R}^{d}), \label{eq:xn_strong_conv} \\ 
    \dot{x}^{(n)} &\rightharpoonup^* \dot{x}  \quad \text{in } L^{\infty}([0,T],\mathbb{R}^{d}). \\  
\end{align}
We know that $x \in C([0,T],\mathbb{R}^{d})$, since $x$ is the uniform limit of continuous functions over a compact interval. Moreover, the embedding $L^{\infty}([0,T],\mathbb{R}^{d}) \subset L^{1}([0,T],\mathbb{R}^{d})$ allows us to convert weak${}^*$ convergence in $L^{\infty}$ into weak convergence in $L^1$, i.e., 
\begin{align}
    \dot{x}^{(n)} \rightharpoonup \dot{x} \quad \text{in } L^1([0,T],\mathbb{R}^d).
\end{align}
Using the Lipschitz continuity of $f$, one can show (see \cite[Lemma 3.6]{Venel2011}) that 
\begin{align}
    f^{(n)} \rightharpoonup f(x(\cdot)) \quad \text{in } L^{1}([0,T],\mathbb{R}^d).
\end{align}
It is now left to show that the limit satisfies the differential inclusion, i.e. $\dot x(t) - f(x(t)) \in -N(S,x(t))$ for a.e. $t \in [0,T]$.

3. By Mazur's Lemma, there exists a sequence $z^{(n)} \in L^1([0,T],\mathbb{R}^d)$ such that 
\begin{align}
    z^{(n)} \to \dot x - f(x(\cdot)) \quad \in L^1([0,T],\mathbb{R}^d) \quad \text{as } n \to \infty,
\end{align}
where $z^{(n)}$ is a convex combination for terms from the sequence $\dot{x}^{(n)} - f^{(n)}$, i.e.,
\begin{align}
    z^{(n)} \in \left\{ \sum_{\ell = n}^\infty \lambda_\ell \left( \dot{x}^{(\ell)} - f^{(\ell)} \right) \mid \lambda_\ell \in [0,1] \text{ s.t. } \sum_{\ell=n}^{\infty} \lambda_\ell = 1 \right\}.
\end{align}
We denote $z(t) \coloneqq \dot{x}(t) - f(x(t))$, which is the limit of $z^{(n)}$ by construction.

By taking another subsequence (without relabeling), we obtain pointwise convergence
\begin{align}
    z^{(n)}(t) \to z(t) \quad \text{for all } t \in \mathcal{T} \quad \text{as } n \to \infty
\end{align}
for a dense set $\mathcal{T} \subseteq [0,T]$.

4. We recall that, by definition, we have the relation
\begin{align}
    \dot{x}^{(n)}(t) - f^{(n)}(t) = \frac{\Phi_{h}( y^{(n)} ) - y^{(n)}}{h} - f(y^{(n)}),
\end{align}
where $y^{(n)} \coloneqq x^{(n)}(\lfloor t \rfloor^{(n)})$ denotes the closest discrete numerical steps for a given value of $t$ and $n$.

Moreover, from \cref{eq:xn_strong_conv}, we get $y^{(n)}(t) \to x(t)$ as $n \to \infty$.
The bounded constraint violation \cref{eq:def:feasible} we also obtain that 
\begin{align}
    d_S(y^{(n)}) \leq K h^{(n)} ,
\end{align}
for some constant $K > 0$.  

For a fixed value $t \in \mathcal{T}$ and an arbitrary test direction $w \in \mathbb{R}^d$ we can compute 
\begin{align}
    \langle z(t), w \rangle 
    &\leq 
    \limsup_{n \to \infty} 
    \langle z^{(n)}(t), w \rangle
    \\ 
    &\leq 
    \limsup_{n \to \infty} \langle \dot{x}^{(n)} - f^{(n)}(t), w \rangle 
    \\ 
    &\leq 
    \limsup_{n \to \infty} \langle \frac{\Phi_{h}( y^{(n)} ) - y^{(n)}}{h} - f(y^{(n)}), w \rangle 
    \\ 
    &\leq 
    \limsup_{n \to \infty} \sigma( -C \partial^P d_S(x(t)), w).
\end{align}
Where we used the constant $C$ from the condition for scalarly upper semicontinuous consistency in \cref{eq:def:consistency}.

Taking the supremum with respect to all test directions $w \in \mathbb{R}^d$ implies 
\begin{align}
    \sigma( \{ z(t) \}, w ) \leq \sigma( -C \partial^P d_S(x(t)) ).
\end{align}
Hence, 
\begin{align}
    z(t) = \dot{x}(t) - f(x(t)) \in -C \partial^P d_S(x(t)) \subseteq -N(S, x(t)).
\end{align}
This shows that the limit of the numerical approximations solves the differential inclusion almost everywhere, hence, concludes the proof.
\end{proof}

\section{Metric calmness of non-overlapping spheres}
\label{app:proof_discs_metric_calm}

We show that the feasible sets of non-overlapping spheres satisfy the conditions for a metrically calm intersection \cref{ass:metric_calm}.

\begin{lemma}
    \label{lem:spheres_metric_calm}
    For given $n, R, d > 0$, let 
    \begin{align} 
    S_{ij} \coloneqq 
    \{ x = (X_1,\dots,X_n) \in \mathbb{R}^{nd} \mid 
    \Vert X_i - X_j \Vert \geq 2 R \} 
    \quad \text{and} \quad 
    S = \bigcap_{1 \leq i < j \leq n} S_{ij}.
    \end{align}
    Then, there exists an $\alpha > 0$ only depending on $n$ and $R$ such that
    \begin{align}
    \inf_{x \in \mathbb{R}^{nd}} \max_{1 \leq i < j \leq n} \frac{d_{S_{ij}}(x)}{d_{S}(x)} \geq \alpha.
    \end{align}
\end{lemma}

The main idea of the proof is visualised in \cref{fig:proof_metric_calm}. Our proof relies on the following two lemmas.

\begin{lemma}
\label{app:lem:max_dS}
    For $n,d,R,S$ defined as in \cref{lem:spheres_metric_calm}, there exists an $\alpha > 0$ (only depending on $n$ and $R$) such that for any $x \in \mathbb{R}^{nd}$ it holds
    \begin{align}
    \alpha d_S(x) \leq R.
    \end{align}
\end{lemma}

\begin{lemma}
\label{app:lem:clusters}
    For any $x \in \mathbb{R}^{nd}$ there exists $1 \leq m \leq n$, $L \leq 2n$ and
    $Q_1,\dots,Q_m \in \mathbb{R}^d$ such that
\begin{align}
\label{eq:app:clusters}
\bigcup_{k=1}^n B_{R}(X_k) \subset \bigcup_{k=1}^m B_{2^{L} R}(Q_k)
\quad \text{and} \quad 
B_{2^{L+1}R}(Q_i) \cap B_{2^{L+1}R}(Q_{j}) = \emptyset
\quad \text{for all } 1 \leq i < j \leq m.
\end{align}
\end{lemma}
\begin{figure}
    \centering
    \includegraphics[width=0.8\textwidth]{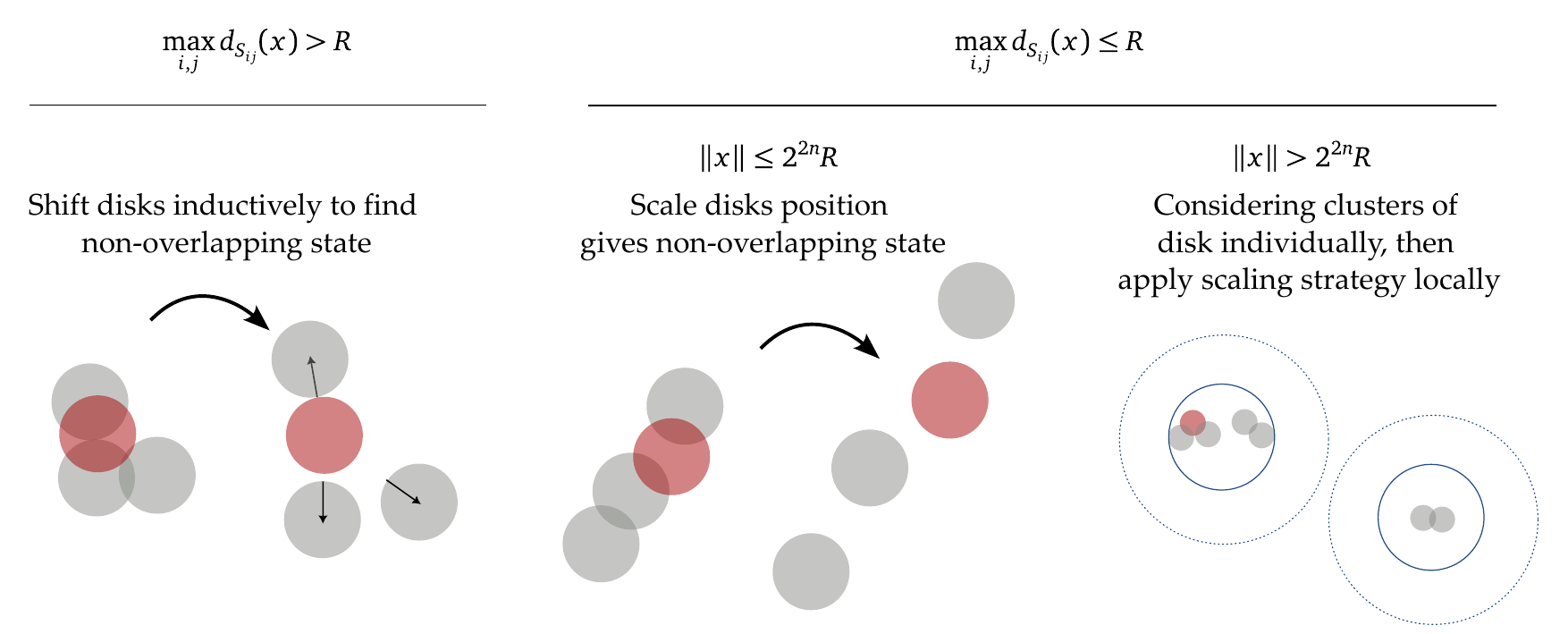}
    \caption{The three main cases in the proof of metric calmness in \cref{lem:spheres_metric_calm}. In all cases, we seek a non-overlapping state $y \in S$ such that $\Vert x - y \Vert \leq \frac{1}{\alpha} \max_{i<j} d_{S_{ij}}(x)$. Hence, finding $y$ can only involve operations for which upper bounds in $\max_{i<j} d_{S_{ij}}(x)$ exist.
If $\max_{i<j} d_{S_{ij}}(x)$ is larger than $R$, then finding a suitable non-overlapping state is possible by consecutively shifting spheres, see \cref{app:lem:max_dS}. However, when all pairwise overlaps are small, then we can resolve all overlaps at once via scaling of the positions by a factor proportional to $\max_{i<j} d_{S_{ij}}(x)$. If spheres are very far away from each other, then \cref{app:lem:clusters} allows us to apply the scaling strategy locally.}
    \label{fig:proof_metric_calm}
\end{figure}

\begin{proof}[Proof of \cref{lem:spheres_metric_calm}]
Let $x = (X_1,\dots,X_n) \in \mathbb{R}^{2n}$ be arbitrary but fixed. In the following, we use the shorthand notation: 
    \begin{align}
    \mathcal{J} \coloneqq \{ (i,j) \mid 1 \leq i < j \leq n \} 
    \quad \text{and} \quad
        D(x) \coloneqq \max_{(i,j) \in \mathcal{J}} d_{S_{ij}}(x).
    \end{align}
We need to show that for all $x\in \mathbb{R}^{2n}$, we have $D(x) \geq \alpha d_S(x)$ for some constant $\alpha$.

If $D(x) > R$ holds, then \cref{app:lem:max_dS} directly yields the claim, since there exists an $\alpha > 0$ (only depending on $n, R$) such that $\alpha d_S(x) \leq R$, hence,
\begin{align}
    D(x) > R \geq \alpha d_S(x).
\end{align}
Therefore, it is only left to show the claim for the case $D(x) < R$.

First, we consider the special case $\Vert x \Vert \leq 2^{L} R$ for some $L \leq 2(n-1)$. We recall that $d_{S_{ij}}(x) = 2R - \Vert X_i - X_j \Vert$ for $x \notin S_{ij}$ and $d_{S_{ij}}(x) = 0$ if $x \in S_{ij}$. Which also implies 
\begin{align}
    \Vert X_{i} - X_j \Vert \geq 2 R - d_{S_{ij}} \geq 2R - D(x).
\end{align}
Next. we define
\begin{align}
\label{app:eq:c_def}
    c \coloneqq \frac{2R}{2R - D(x)} \in [1,2)
\end{align}
and compute
\begin{align}
    \Vert c X_i - c X_j \Vert &= c \Vert X_i - X_j \Vert \\
    &\geq c ( 2 R - D(x) ) \\ 
    &= 2R \quad \text{for all } (i,j) \in \mathcal{J}.
\end{align}
This implies $cx \in S$, which already yields the claim since, 
\begin{align}
    d_S(x) &\leq \Vert c x - x \Vert 
    = 
    \vert c - 1 \vert \Vert x \Vert \\ 
    &= \frac{D(x)}{2R - D(x)} \Vert x \Vert \\
    &\leq \frac{2^{2(n-1)}R}{R} D(x) = 2^{2(n-1)} D(x).
\end{align}
Which is the claim for $\alpha = (2^{2(n-1)})^{-1}$. 

Now, we consider the unbounded case $x \in \mathbb{R}^{nd}$ with $D(x) < R$. By \cref{app:lem:clusters}, we know that there exist $1 \leq m \leq n$, $Q_1,\dots,Q_m \in \mathbb{R}^{d}$ and $L \leq 2(n-1)$ such that
\begin{align}
\bigcup_{k=1}^n B_{R}(X_k) \subset \bigcup_{k=1}^m B_{2^{L} R}(Q_k)
\quad \text{and} \quad 
B_{2^{L+1}R}(Q_i) \cap B_{2^{L+1}R}(Q_{j}) = \emptyset
\quad \text{for all } 1 \leq i < j \leq m.
\end{align}
The lemma provides independent clusters of positions that are sufficiently separated to apply the same scaling as in the previous case without introducing new intersections. To formalize this, we define an index partition via
\[
I_k \coloneqq \{ i \mid 1 \leq i \leq n, X_i \in B_{2^L R}(Q_k) \}.
\]
In the following, we will scale the points with indices in $I_k$ by the same factor $c$ as defined in \cref{app:eq:c_def}, but this time with the origin $Q_k$. This yields
\begin{align}
    y = (Y_1,\dots,Y_n) \coloneqq q + c(x-q)
    \quad 
    \text{where} 
    \quad 
    q \coloneqq (Q_1,\dots,Q_m).
\end{align}
For any $i, j \in I_k$, we compute
\begin{align}
    \Vert Y_i - Y_j \Vert 
    &= \Vert Q_k - c(X_i - Q_k) - Q_k + c(X_j - Q_k) \Vert \\ 
    &= c \Vert X_i - X_j \Vert \\ 
    &\geq c (2R - D(x)) \\ 
    &= 2R.
    \label{app:eq:ij_in_Ik}
\end{align}
On the other hand, for $i \in I_k$ and $j \in I_{\ell}$ with $k \neq \ell$ we use the first and second conditions in \cref{eq:app:clusters} (and $1 \leq c < 2$) to compute 
\begin{align}
    \Vert Y_i - Y_j \Vert 
    &= \Vert Q_k - c(X_i - Q_k) - Q_\ell + c(X_j - Q_\ell) \Vert 
    \\
    &\geq \Vert Q_k - Q_\ell \Vert - c \Vert X_i - Q_k \Vert - c \Vert X_j - Q_\ell \Vert \\ 
    &\geq 2^{L+2} R - 2c (2^L R - R) \\
    &\geq 2^{L+2}R - 2^{L+2}R + 2cR \\ 
    &\geq 2R. \label{app:eq:ij_ij_Ik_Il}
\end{align}
Combining \cref{app:eq:ij_in_Ik,app:eq:ij_ij_Ik_Il} yields $y \in S$. We obtain the claim by computing 
\begin{align}
    d_S(x) 
    &\leq \Vert x - y \Vert \\ 
    &= \Vert x - q - c(x - q)\Vert \\ 
    &= \vert 1-c \vert \Vert q - x \Vert \\ 
    &\leq \frac{D(x)}{2R - D(x)} \sum_{k=1}^m \sum_{i \in I_k} \Vert X_i - Q_k \Vert 
    \\ 
    &\leq \frac{n 2^L R}{R} D(x) \leq n 2^{2(n-1)} D(x).
\end{align}
Which yields the claim also for the unbounded case with $\alpha = (n 2^{2n})^{-1}$.

\end{proof}

\begin{proof}[Proof of \cref{app:lem:max_dS}]
We consider an arbitrary $x = (X_1,\dots,X_n) \in \mathbb{R}^{nd}$. 
Since the norms $x \mapsto \Vert x \Vert_{\mathbb{R}^{nd}}$ and $x = (X_1,\dots,X_n) \mapsto \sum_{i=1}^n \Vert X_i \Vert_{\mathbb{R}^2}$ are equivalent (with constants only depending on $n$), it suffices to show that there exist $y \in \mathbb{R}^{nd}$ such that 
\begin{align} 
\label{app:eq:metric_case1_inductionclaim}
y \in S \quad \text{and} \quad 
\sum_{i=1}^n \Vert X_i - Y_i \Vert \leq 2 R \frac{n(n-1)}{2}.
\end{align} We prove this by induction over the number of spheres.

Let us define $x_k = (X_1,\dots,X_k) \in \mathbb{R}^{2k}$ as the truncation of $x$ to the first $k$ center positions. 
 
For $k=1$, the claim holds trivially, as $x_1 = (X_1)$ represents only one sphere, which is a non-overlapping state and $\Vert X_1 - X_1 \Vert = 0 \leq 2R \frac{1(1-1)}{2}$.

Assuming the induction assumption for $k \geq 1$, we have that there exist positions $\tilde{x}_k = (\tilde{X}_1,\dots,\tilde{X}_k)$ such that $\Vert \tilde{X}_i - \tilde{X}_j \Vert \geq 2R$ for all $1 \leq i < j \leq k$ and $\sum_{i=1}^k \Vert X_i - \tilde{X}_i \Vert \leq 2R \frac{k (k-1)}{2}$.

We construct an non-overlapping state $y_{k+1} = (Y_1,\dots,Y_{k+1})$ via
\begin{align}
    Y_i \coloneqq \tilde{X}_i + 2 R \frac{\tilde{X}_i - X_{k+1}}{\Vert \tilde{X}_i - X_{k+1} \Vert}, 
    \quad 
    \text{and}
    \quad 
    Y_{k+1} \coloneqq X_{k+1}.
\end{align}
If the singularity $\tilde{X}_i = X_{k+1}$ occurs for some $i \leq k$, then we replace the fraction with an arbitrary normalized vector.
Since the first $k$ positions are pushed away from the position $X_{k+1}$, we obtain  
\begin{align} 
\Vert Y_i - Y_j \Vert 
\geq \Vert \tilde{X}_i - \tilde{X}_j \Vert \geq 2R \quad \text{ for all } 
1 \leq i < j \leq k.
\end{align}
Moreover, we also have $\Vert Y_i - Y_{k+1}\Vert \geq 2R$ for all $1\leq i \leq k$.

Therefore, the spheres with center positions $y$ are non-overlapping and satisfy ${\sum_{i=1}^{k+1} \Vert X_i - Y_i \Vert \leq 2R \frac{k(k-1)}{2} + 2R k}$, which concludes the induction proof. 
\end{proof}

\begin{proof}[Proof of \cref{app:lem:clusters}]
The main idea of this proof is to construct $m$ and $Q_1,\dots,Q_m$ with an iteration that either terminates in a state satisfying the claim. 
We start with $m=n$,  $Q_1^{(0)} \coloneqq X_1, \dots, Q_n^{(0)} \coloneqq X_n \in \mathbb{R}^d$ and $L = 0$, which ensures that the first condition in \eqref{eq:app:clusters} is satisfied. If the second condition is also satisfied, we are done.
    
Otherwise, we will continue inductively by constructing new points wit the following procedure.

Given $Q_1^{(k)}, \dots, Q_m^{(k)} \in \mathbb{R}^d$ and $L \geq 0$ such that 
\begin{align}
    \bigcup_{i=1}^n B_{R}(X_i) \subset \bigcup_{j=1}^m B_{2^{L} R}(Q_j^{(k)})
\end{align}
holds, but the second condition in \cref{eq:app:clusters} is not satisfied. Then, without loss of generality, there is an index $1 \leq i^* < m$ such that
\begin{align}
    B_{2^{L+1} R}(Q_{i^*}^{(k)}) \cap B_{2^{L+1} R}(Q_m^{(k)}) \neq \emptyset 
    \quad 
    \text{which is the same as}
    \quad 
    \Vert Q_{i^*}^{(k)} - Q_m^{(k)} \Vert 
    \leq 2^{L+1} R.
    \label{app:eq:cluster_intersecting}
\end{align}
We then remove the point $Q_m^{(k)}$, i.e., we define $Q_{j}^{(k+1)} \coloneqq Q_{j}^{(k)}$ for $1 \leq j \leq m-1$ and set $m' \coloneqq m-1$ and $L' \coloneqq L+2$ for the next step.

Due to \eqref{app:eq:cluster_intersecting} we have 
 \begin{align} 
B_{2^{L} R}(Q_m^{(k)}) \subset B_{2^{L+2} R}(Q_{{i^*}}^{(k+1)})
\end{align}
which implies 
\begin{align}
    \bigcup_{i=1}^n B_{R}(X_i) \subset \bigcup_{j=1}^{m-1} B_{2^{L+2} R}(Q_j^{(k+1)}).
\end{align}
If the new choice $m', Q_1^{(k+1)}, \dots, Q_{m'}^{(k+1)}$ and $L'$ satisfy \cref{eq:app:clusters} then we are done. Otherwise, we repeat the procedure until $m=1$ and $L = n-1$. Then, the second condition in \cref{eq:app:clusters} holds trivially (since there is only one set left). Hence, the construction shows the existence of $m,Q_1,\dots,Q_m$ and $L \leq 2(n-1)$ satisfying \cref{eq:app:clusters}.
\end{proof}

\end{appendices}

\bibliography{plunder_main.bib}

\end{document}